\documentclass[11pt,reqno]{amsart}
\usepackage[colorlinks,citecolor=green,linkcolor=red]{hyperref}
\usepackage{amsmath}
\usepackage{xfrac}
\usepackage{stmaryrd,mathrsfs,bm,amsthm,mathtools,yfonts,amssymb,color,enumitem}
\usepackage{xcolor}
\usepackage{physics}
\usepackage[margin=0.7in]{geometry}

\begingroup
\newtheorem{theorem}{Theorem}[section]
\newtheorem{thm}[theorem]{Theorem}
\newtheorem{lem}[theorem]{Lemma}
\newtheorem{proposition}[theorem]{Proposition}
\newtheorem{con}[theorem]{Conjecture}
\endgroup

\theoremstyle{definition}
\newtheorem{definition}[theorem]{Definition}

\newtheorem{remark}[theorem]{Remark}

\numberwithin{equation}{section}

\DeclareMathOperator{\supp}{spt}
\newcommand{\reg}{{\rm Reg}}
\DeclareMathOperator{\Div}{div}
\newcommand\lap{{\Delta}}
\newcommand\LL{{\mathcal{L}}}
\newcommand\res{\mathop{\hbox{\vrule height 7pt width .3pt depth 0pt
			\vrule height .3pt width 5pt depth 0pt}}\nolimits}
\newcommand\cH{{\mathcal{H}}}

\newcommand{\bE}{{\mathbf{E}}}

\newcommand{\bC}{{\mathbf{C}}}

\newcommand{\be}{{\mathbf{e}}}

\newcommand\N{{\mathbb N}}

\newcommand{\eps}{{\varepsilon}}
\newcommand{\bmax}{{\mathbf{m}}}

\def\XXint#1#2#3{{\setbox0=\hbox{$#1{#2#3}{\int}$ }
		\vcenter{\hbox{$#2#3$ }}\kern-.6\wd0}}

\newcommand{\Lip}{{\rm {Lip}}}
\newcommand{\diam}{{\rm {diam}}}
\newcommand{\dist}{{\rm {dist}}}

\newcommand\Id{{\rm Id}}
\newcommand{\sff}{\mathrm{II}}

\newcommand{\cG}{{\mathcal{G}}}
\newcommand{\Iq}{{\mathcal{A}}_Q}
\def\a#1{\left\llbracket{#1}\right\rrbracket}
\newcommand{\de}{\partial}
\newcommand{\R}{\mathbb{R}}
\newcommand{\mres}{\res}
\newcommand\B{{\mathbf{B}}}

\newcommand{\defeq}{\coloneqq}
\newcommand{\eqdef}{\eqqcolon}

\def\XXint#1#2#3{{\setbox0=\hbox{$#1{#2#3}{\int}$}
		\vcenter{\hbox{$#2#3$}}\kern-.5\wd0}}

\let\epsilon\varepsilon
\let\phi\varepsilon

\title[$C^\infty$ rectifiability of stationary varifolds]{$C^\infty$ rectifiability of stationary varifolds}
\author{Camillo Brena}
\address{School of Mathematics, Institute for Advanced Study, 1 Einstein Dr., Princeton NJ 08540, USA}
\email{cbrena@ias.edu}
\author{Camillo De Lellis}
\address{School of Mathematics, Institute for Advanced Study, 1 Einstein Dr., Princeton NJ 08540, USA}
\email{camillo.delellis@ias.edu}
\author{Federico Franceschini}
\address{School of Mathematics, Institute for Advanced Study, 1 Einstein Dr., Princeton NJ 08540, USA}
\email{ffederico@ias.edu}

\begin{document}

\begin{abstract}
		In this paper we prove that, for every integers $m\geq 2$ and $n\geq 1$, the support of any stationary $m$-dimensional integer rectifiable varifold $V$ in an open set $U\subset \mathbb R^{m+n}$ is $C^\infty$ rectifiable, namely it can be covered, up to an $\mathcal{H}^m$-null set, with countably many $C^\infty$ $m$-dimensional graphs. 
\end{abstract}

\maketitle
\tableofcontents
   
\section{Introduction}

Integer rectifiable $m$-dimensional varifolds, which in this note will be denoted by $V$, in an open set $U\subset \mathbb R^{m+n}$ can be defined as Radon measures 
\begin{equation}\label{e:varifold-rett}
\Theta\, \mathcal{H}^m \res E\, ,
\end{equation}
where $E\subset U$ is an $m$-dimensional rectifiable set, $\mathcal{H}^m$ denotes the Hausdorff $m$-dimensional measure, and $\Theta$ is a Borel function taking positive integer values ($\mathcal{H}^m$-a.e.); cf.\ \cite{DeLellisAllard,Simon}. The theory was pioneered by Almgren and especially Allard, see \cite{AllardFirst}. In the most general setting, first proposed by L.C. Young, a varifold is a nonnegative Radon measure on the Grassmannian of unoriented $m$-dimensional planes $G_m (U) = U \times G_m (\mathbb R^{m+n})$: with the latter definition the varifold is rectifiable if its marginal, denoted by $\|V\|$, on $U$ takes the form \eqref{e:varifold-rett} and its disintegration on the fibers $\{x\}\times \mathbb R^{m+n}$ (with $x\in E$) consists of the Dirac mass on the approximate tangent $T_x E$ to $E$. Since we will always consider only rectifiable varifolds, we can ignore these details; however, following the notation of Almgren and Allard, the Radon measure in \eqref{e:varifold-rett} will be denoted by $\|V\|$.

We will assume that $V$ is stationary in $U$: this means that, for every given $X\in C^\infty_c (U, \mathbb R^{m+n})$, if we let $\Phi_t$ be the one-parameter family of diffeomorphisms of $U$ generated by $X$, then
\begin{equation}\notag
\delta V (X) \defeq \left.\frac{d}{dt}\right|_{t=0} \|(\Phi_t)_\sharp V\| (U) = 0\, .
\end{equation}
Here $\psi_\sharp V$ denotes the varifold $(\psi (E), \Theta\circ \psi^{-1})$ when $\psi$ is a $C^1$ diffeomorphism.
It follows from Allard's monotonicity formula that, since $V$ is integral and stationary, we can assume that, without loss of generality, $E$ is $\supp(\|V\|)\cap U$ and $\Theta$ is given pointwise by the upper semicontinuous function
\begin{equation}\notag
\Theta (V, x) = \lim_{r\downarrow 0} \frac{\|V\| (B_r (x))}{\omega_m r^m}\, .
\end{equation}
This gives a ``canonical pair'' $(E, \Theta)$ and rids us of any tedious discussion of $\mathcal{H}^m$-null sets.

Allard's classical interior regularity theory, developed in the pioneering work \cite{AllardFirst}, implies that a stationary (integral) varifold is always regular in a relatively dense and open subset $\reg (V) \subset \supp (V)$. This means that for every $x\in \reg (V)$ there is a ball $\B_r (x)$ with the property that $\supp (V) \cap \B_r (x)$ is a smooth connected submanifold $\Sigma$ and $\|V\|\res \B_\rho (x) = \Theta_0 \mathcal{H}^m \res \Sigma$ for an appropriate integer costant $\Theta_0\geq 1$. Even though 53 years have passed since the appearance of Allard's work we still do not know whether the complement of $\reg (V)$ is $\mathcal{H}^m$-null (a natural expectation, given the known examples and some recent partial results like \cite{HS}, is in fact that such complement has Hausdorff dimension $m-1$). In \cite{CCSvarifolds} the first two authors and Stefano Decio proposed, as a first step towards the proof of the latter result, the following conjecture (see \cite[Conjecture~1.4]{CCSvarifolds}). 

\begin{con}\label{c:flat-high-order}
Assume $V$ is a stationary $m$-dimensional varifold in $U\subset \mathbb R^{m+n}$ and let $x_0\in U$ be a point at which $\supp (V)$ has an approximate tangent and
\begin{equation}\notag
\lim_{r\downarrow 0} \frac{\|V\| (\{\Theta(V,\,\cdot\,)\neq \Theta (V, x_0)\}\cap \mathbf{B}_r (x_0))}{r^m} = 0\, .
\end{equation}
Then there is a smooth classical minimal $m$-dimensional graph $\mathcal{M}$ in some neighborhood of $x_0$ with the property that 
\begin{equation}\notag
\int_{\mathbf{B}_r (x_0)} \dist (x, \mathcal{M})^2 \,\dd\|V\| (x) = o (r^N)\qquad \mbox{for every $N\in \mathbb N$\,.}
\end{equation}
\end{con}

In this paper we will prove a partial result in the above direction, in fact we will first prove the following ``$C^\infty$ rectifiability'' result. 

\begin{thm}[Smooth rectifiability]\label{rectthm}
Let $V$ be an $m$-dimensional  stationary varifold in an open set $U\subset \mathbb R^{m+n}$. Then $V$ is $C^{\infty}$-rectifiable. Namely, there exist countably many $C^{\infty}$ maps $f_k:\R^m\supset B_1^m\rightarrow U$ such that 
\begin{equation*}
    \|V\|\Big(U\setminus \bigcup_k f_k(B_1^m)\Big)=0\, .
\end{equation*}
\end{thm}

But as a consequence of our approach we can also show the following.

\begin{thm}\notag
Assume $V$ and $x_0$ satisfy the assumptions of Conjecture \ref{c:flat-high-order}. Then there is a smooth classical (not necessarily minimal) $m$-dimensional graph $\mathcal{M}$ in some neighborhood of $x_0$ with the property that 
\begin{equation}\notag
\int_{\mathbf{B}_r (x_0)} \dist (x, \mathcal{M})^2 \,\dd\|V\| (x) = o (r^N)\qquad \mbox{for every $N\in \mathbb N$\,.}
\end{equation}
\end{thm}
This can be seen as an answer to a  weaker (as  we are not able to prove that $\mathcal{M}$  is minimal!) version of Conjecture \ref{c:flat-high-order}.

Before coming to a description of the strategy of our proof we wish to recall what is already known in the literature. First of all, it was pointed out by Brakke in \cite{Brakke} that, in combination with Almgren's theory of multivalued function, Allard's approach in \cite{AllardFirst} already implies the $C^{1,\alpha}$ rectifiability for every $\alpha<1$: the submanifolds of the covering of the rectifiable set $\supp (V)$ can actually be chosen to be $C^{1,\alpha}$. A much more interesting development, due to Menne in \cite{MenneJGA}, is the $C^2$-rectifiability. In fact Menne's theorem applies to varifolds with bounded mean curvature and it is thus an optimal statement in his context. 

\subsection{Acknowledgments} 
 The first author is supported by the National Science Foundation under Grant No.\ DMS-1926686.
The third author gratefully acknowledges support from the Giorgio and Elena Petronio Fellowship while working on this project.

The authors thank G.\ De Philippis for discussions around the topic of Section \ref{vefdscxzcads}.
\section{Strategy of the proof}

The most important idea of the proof is to modify the classical Allard--De Giorgi decay lemma. The latter shows that, under the hypothesis that in a given ball $\B_\rho$ the density is constant and the (appropriately scaled) $L^2$ distance of the varifold to the best approximating plane $\pi$ falls below a sufficiently small threshold $\kappa>0$, then in a smaller ball with a proportional radius the same quantity is (less than) a fraction $\gamma$ of what it is in $\B_\rho$. This ensures that such $L^2$ distance decays like a power law at many points, in turn giving the $C^{1,\alpha}$ regularity of the varifold in a neighborhood of each such point. From this Allard concludes the regularity of the varifold in a dense open subset of its support. The obstruction in showing that the singular set (the complement of the regular set) is $\mathcal{H}^n$ null is in the assumption about the density. Brakke later pointed out in \cite{Brakke} that the assumption on the density can be suitably weakened if Allard's strategy is combined with Almgren's theory of multivalued functions: it is possible then to reach the same decay conclusion but the nature of the ``new'' density assumption only allows to {\em cover} $\mathcal{H}^m$-almost all the support of the varifold with countably many $C^{1,\alpha}$ submanifolds, hence concluding its $C^{1,\alpha}$ rectifiability (see \cite{CCSvarifolds} for a fresh take on the latter idea). 

\medskip

Our main point is to substitute the ``best approximating plane $\pi$'' with the ``best approximating classical minimal surface'' in the rescaled $L^2$ distance.  Using all minimal surfaces we are able to reach a decay parameter $\gamma$ arbitrarily small, while keeping the ratio between the scales fixed. We need to lower the threshold $\kappa$ depending on $\gamma$ and we need to adjust accordingly the assumption on the density. The decay lemma which we are able to reach is still good enough to show that at $\mathcal{H}^m$-a.e.\ point of the support the scaled $L^2$ distance to the ``best minimal graph'' decays faster than any power. In particular classical PDE estimates on (sufficiently flat) smooth minimal surfaces and a standard decomposition trick allows us to prove Theorem \ref{rectthm}. 

The idea that planes can be substituted with a larger class of models to gain directly more regularity is not new in the literature. For instance it has been used by Savin in \cite{SavinAllard}: in the latter reference the best approximating plane is substituted with the ``best approximating parabola'', allowing him to get directly to $C^2$ regularity for his notion of viscosity solution of the minimal surface system. In this work we push the latter idea to the extreme, reaching directly $C^\infty$ regularity.

\medskip

In order to implement our strategy we need two technical tools which might be of independent interest. The first is a multivalued Lipschitz approximation \`a la Almgren (for varifolds) on the normal bundle of a $C^2$ submanifold $M$; our argument for this is a suitable modification of the one we give in \cite{CCSvarifolds} for the Almgren's approximation, which is alternative to Almgren's original approach in \cite{Almgren00}, adopted subsequently in \cite{Brakke} and \cite{MenneJGA}. The second is a generalization of the Allard's ``tilt-excess inequality'', for which we refer to \cite[Lemma 8.13]{AllardFirst}, \cite[Lemma 22.2]{Simon}, and \cite[Proposition 4.1]{DeLellisAllard}. The original tilt excess inequality of Allard can be understood as the counterpart of the well-known Caccioppoli inequality for solutions of elliptic PDEs, provided we think of the varifold as a suitable measure-theoretical generalization of a classical minimal graph. Our version can then be understood as the Caccioppoli-type inequality which we would conclude if it were possible to parametrize our stationary varifold as a graph over the normal bundle of a reference minimal surface $M$. In the proof we take advantage of ideas in the recent work \cite{DPGS}.   

\subsection{Decay lemma} In this section we state the decay lemma which is at the core of our work. $\pi^\perp$ denotes the orthogonal complement of the $m$-dimensional linear subspace $\pi$, while $B_r (x, \pi) = \B_r (x)\cap (x+\pi)$ and $\bC_r (x,\pi) = B_r (x, \pi) + \pi^\perp$. $\pi$ and $x$ will be dropped if they are (respectively) $\pi_0= \R^m\times \{0\}$ and the origin. We will moreover use the shorthand notation $d_E$ for the distance to a set $E$ and $d_V$ as a shorthand notation for $d_{\supp (V)}$.

\begin{definition}[$\delta$-flatness]\label{e:delta-flatness}
Let $M$ be an $m$-dimensional submanifold of $\bC_r\subset\R^{m+n}$. We say that $M$ is $\delta$-flat in $\bC_r$ if $M\cap\bC_r$ is the the graph of a map $g\in C^2(B_r,\R^n)$ that satisfies the scale-invariant estimates
\begin{equation}\notag
    \sup_{x\in B_r} r^{-1}|g(x)|+|Dg(x)|+r|D^2 g(x)|\le \delta\,.
\end{equation}
\end{definition}


\begin{lem}[Decay lemma]\label{decaylemma}
For every triple of integers $m, n, Q\geq 1$ there are positive constants $\eta_0$, $\eps_0$, $\delta_0$, and $C$ with the following property.
Let $V$ be an $m$-dimensional stationary varifold in $\bC_{100}$ and $M$ a $\delta$-flat minimal $m$-dimensional submanifold in $\bC_{200}$ such that $\delta \leq \delta_0$,
\begin{align}
& \supp(V)\cap \bC_{100}\subset \{(x,y):|y|\le 1\}\, ,\label{vrefdsasvd}\\
& Q- \textstyle{\frac{1}{2}}<\frac{\|V\|(\bC_{30})}{\omega_m30^m} <Q+\textstyle{\frac{1}{2}}\, , \label{vrefdasvd-2}\\
&\frac{\|V\|(\bC_{100})}{\omega_m100^m} \le Q+\textstyle{\frac{1}{2}}\, , \label{vrefdvasvd-3}\\
& \boldsymbol{\eta} \defeq \frac{\|V\|(\{\Theta_V<Q\}\cap \bC_{100})}{\omega_m100^m}\le \eta_0\, ,\quad\mbox{and}\label{bvtewfadsvcasdc}\\   
&\boldsymbol{\eps} \defeq \Big(\frac{1}{\omega_m100^{m+2}}\int_{\bC_{100}} d^2_M  \,\dd\|V\|\Big)^{1/2} \le \eps_0\, .
\end{align}
Then, there exists a $m$-dimensional minimal surface $M',$  $\delta'$-flat in $\bC_{2}$, such that
    \begin{equation}\label{eq:contraction}
    \Big(\frac{1}{\omega_m}\int_{\bC_{1}} d^2_{M'} \, \dd\|V\| \Big)^{1/2}\le C(\sqrt{\boldsymbol{\eta}}+\sqrt{\boldsymbol{\eps}})\boldsymbol{\eps}\quad \text{and} \quad\delta'\le \delta+C\boldsymbol{\eps}\,.
    \end{equation}
\end{lem}

The assumptions of Lemma \ref{decaylemma} imply in particular the following two facts, which are not directly related with the decay of the excess.
First, not only $V$  is close to $M$, but also  $M$ is close to $V$, in the $L^2$ sense. In particular, $V$ ``has no holes'', in an integral sense. This is stated in \eqref{verfdsavdsc1}. Second, whenever we have another $\delta$-flat minimal surface $M''$ that satisfies a property similar to the one in \eqref{eq:contraction}, we can bound the distance of $M''$ to $M$. This is \eqref{verfdsavdsc2} and will play a key role in the derivation of the Whitney estimates that we use to show $C^\infty$-rectifiability.
\begin{lem}\label{verfdsavdsc}
    Under the assumptions of Lemma \ref{decaylemma}, the following holds.
    \begin{itemize}
        \item There is a constant $C$ which depends only on $m$, $n$ and $Q$ such that
    \begin{equation}\label{verfdsavdsc1}
       \Big( \frac{1}{\omega_m }\int_{M\cap\bC_{1}}d_{V}^2\,\dd\mathcal{H}^m \Big)^{1/2}\le C\boldsymbol{\eps}\, .
    \end{equation}
    \item Let $x$ with $|x|<2$ and let $M''$ be an $m$-dimensional minimal surface which is $\delta_0$-flat in $\bC_2(x)$ and 
    \begin{equation}\notag
        \Big(\frac{1}{\omega_m}\int_{\bC_{1}(x)} d^2_{M''} \, \dd\|V\| \Big)^{1/2}\le \sigma\,,
    \end{equation}
    for a $\sigma>0$  small  enough, depending on $m$, $n$ and $Q$.
    If we write $M$ as the graph of $g$ and $M''$ as the graph of $g''$ (over $B_2$), then 
    \begin{equation}\label{verfdsavdsc2}
        \|D^l(g-g'')\|_{L^\infty(B_{1/10}(x))}\le C_l(\eps+\sigma)\qquad\text{for every $l\ge 0$}\,,
    \end{equation}
    where $C_l$ is a constant that depends only upon $m$, $n$,  $Q$ and $l$.
        \end{itemize}
\end{lem}
\begin{remark}\label{portoavantiipotesi}
    Under the assumptions of the Decay Lemma it is also true that \eqref{vrefdsasvd}, \eqref{vrefdasvd-2}, and \eqref{vrefdvasvd-3}, holds at the smaller scale, i.e.
    \begin{align*}
    &\supp(V)\cap \bC_{1}\subset \{(x,y):|y|\le 1/100\}\, ,\\
    &Q-\textstyle{\frac{1}{2}}<\frac{\|V\|(\bC_{3/10})}{\omega_m(3/10)^m} <Q+\textstyle{\frac{1}{2}}\, ,\\ 
    &\frac{\|V\|(\bC_{1})}{\omega_m} \le Q+\textstyle{\frac{1}{2}}\,.
    \end{align*}
    Thus the only assumption which is not ``carried over at the smaller scale'' is \eqref{bvtewfadsvcasdc}. Notie however that \eqref{bvtewfadsvcasdc} is always satisfied if $Q=1$.
\end{remark}

\subsection{Elliptic estimates for the distance from a minimal surface}\label{vefdscxzcads} In this section we state our generalization of Allard's Tilt-excess inequality. A key tool are some properties of the distance function to a minimal surface. 
We will show, in particular, that, if $M$ is a $\delta$-flat minimal surface, then the function $d^2_M$ satisfies
\[
\lap_L  \tfrac 12 d_M^2(z) + \delta^2 d^2_M(z) \ge \tfrac14|L-T_{p(z)}M|^2\,,
\]
where $L$ is any $m$-dimensional subspace and $z\in \R^{m+n}$. As it is well-known, this type of elliptic inequality implies energy estimates as well as height bounds. We state these conclusions in the following two propositions.

\begin{proposition}[Tilt-excess inequality]\label{prop:caccioppoli}
There are constants $C$ and $\delta_0$ which depends only on $m$ and $n$ with the following property. Assume that:
\begin{itemize}
\item[(i)] $M$ is an $m$-dimensional smooth submanifold, $\delta$-flat in $\bC_{2}$ with $\delta \leq \delta_0$, and denote by $p\colon B_1 (0, \pi_0)+ B_1 (0, \pi_0^\perp) \to M$ be the closest-point projection onto $M$. 
\item[(ii)] $V$ is an $m$-dimensional stationary varifold in $\bC_1$, with $\supp(V)\subset B_1 (0, \pi_0)+ B_1 (0, \pi_0^\perp)$. 
\end{itemize}
Then, $|T_zV-T_{p(z)}M|^2\ge | D_V d_M|^2(z)$, and 
\begin{equation}\label{eq:caccioppoli}
    \int_{\bC_{  1/2}} |T_zV-T_{p(z)}M|^2\, \dd\|V\|(z)\le C \int_{\bC_1}d_M^2+d_M|H_M\circ p| \, \dd\|V\|\, .
\end{equation}
\end{proposition}
\begin{proposition}[Height bound]\label{prop:altezza}
    Under the same assumptions of Proposition \ref{prop:caccioppoli},
        \begin{equation}\label{eq:L2Linf}
        \sup_{\supp(V)\cap \bC_{1/2}}d_M^2\le C\Big(\int_{\bC_1} d_M^2\, \dd\|V\|+\|V\|(\bC_1)\sup_{\bC_1\cap M}|H_M|^2\Big)\,,
    \end{equation}
        where $C$ depends on $m$ and $n$.
\end{proposition}

\begin{remark}
    Testing the stationarity of $V$ with the vector field $\zeta^2|z|^{2-m}  D(d^2)-d^2 D(|z|^{2-m}\zeta^2)$, one can actually replace the left-hand side of \eqref{eq:caccioppoli} with
    \[
    \int_{\bC_{  1/2}} |z|^{2-m}\, |T_zV-T_{p(z)}M|^2\, \dd\|V\|(z)\,.
    \]
\end{remark}

\subsection{Lipschitz approximation on the normal bundle} In this section we state the Lipschitz approximation theorem for stationary varifolds on the normal bundle of sufficiently flat $C^2$ submanifolds of the same dimension. 
From now on we will use the notation $B^M_r$ for the geodesic ball of $M$ with radius $r$ centered at $o\defeq M\cap(\{0\}\times \R^n)$.

  \begin{thm}[Lipschitz approximation]\label{Lipapprox} For any triple of integers $m,n,Q\geq 1$ there are constants $\varepsilon_0$, $\delta_0$, and $C$ with the following properties. Assume that
  \begin{itemize}
      \item[(i)]  $M$ is a smooth $m$-dimensional submanifold, $\delta$-flat in $\bC_{10}$, with $\delta \leq \delta_0$;
      \item[(ii)] $V$ is an $m$-dimensional stationary varifold in $\bC_5$, with 
      \begin{align}
      & \frac{\|V\|(\bC_5)}{\omega_m 5^m}\le Q+1\, ,\\
      &\supp V\subset\{(x,y):|x|<5,|y|<1\}\, ,\\ 
      &\bE\defeq \frac{1}{\omega_m 5^m}\int_{\bC_{5}} |T_x V-T_{p_M(x)}M|^2 \, \dd\|V\|(x) \leq \varepsilon_0\, , \quad \mbox{and}\\
	   &\frac{\|V\|(\bC_{3})}{\omega_m 3^m}  \in(Q-\textstyle{\frac{1}{2}},Q+\textstyle{\frac{1}{2}})\,.\label{vfeadsc}
      \end{align}
  \end{itemize}
  Then, for any given $\lambda \in (0,1)$,  there exist a $Q$-valued  Lipschitz function $f: B_1^M\rightarrow\Iq(\R^{m+n})$ and a Borel set $K\subset B_1^M$ such that 
  \begin{itemize}
      \item[(a)] $f(x)\in \Iq(T_x ^\perp M)$ for every $x\in B_1^M$;
      \item[(b)] The following estimates hold:
    \begin{align}
    &\Lip (f) \le C(\lambda+\delta^2)^{\sfrac{1}{2m}}\label{Lipdcsa}\,, \\
	&|B_1^M\setminus K|+\|V\|\big(p_M^{-1}(B_1^M\setminus K)\big)\le \frac{C}{\lambda}\bE\, , \label{measureestimate}\\
    &\sup_{x\in B_1^M}\cG(f(x),Q\a{0})\le C \sup_{x\in p^{-1}_M(B_1^M)} d_M(x)\, \label{cdsacads};
    \end{align}
    \item[(c)] for every $x\in K$, if $f(x)=\sum_h Q_h\a{p_h}$, then $\supp(V)\cap p^{-1}_M(x) = \bigcup_h \{x+p_h\}$; 
    \item[(d)] for every $x\in K$ and every $z\in p^{-1}_M(x)\cap\supp(V)$
    we have 
    \begin{equation}\label{densdvacs}
		\Theta(z)=\sum_{h:p_h=y}Q_h\,.
	\end{equation}
    \end{itemize}
\end{thm}

\section{Preliminaries}

In this section we collect some preliminary notions and we state some results which are going to be useful throughout the paper. Their proofs will be deferred to the Appendix. 

\subsection{The geometric set-up}\label{vfsadcasddsvcca}

From now on we will always assume that $M$ is an $m$-dimensional submanifold (of class at least $C^2$), which is $\delta_0$-flat in $\bC_{2}$. 
We will moreover fix an orthonormal frame $\{e_i(x)\}_{i=1,\ldots,m+n}$, with $C^1$ dependence on $x\in M$, where the first $m$ vectors are tangent to $M$ and the last $n$ vectors are normal to $M$. Fixing such a choice is convenient even though all the quantities which we define will in fact be independent of it.

The second fundamental form of $M$ at $x$ is the bilinear map $\sff \colon T_xM\times T_xM\to T_xM^\perp$ defined by $\sff(e,e')=(D_e e')^\perp$. Moreover, if $\nu (x)$ is a vector normal to $T_x M$, we define
\begin{equation*}
    \sff_\nu(e,e')\defeq (D_e e')\cdot \nu =-e'\cdot D_e\nu=-e\cdot D_{e'}\nu\qquad \text{for }e,e'\in T_x M\text{ and } \nu\in T_x^\perp\,,
\end{equation*}
where, for the second equality, we fix any extension of $\nu(x)$ to a neighborhood of $x$ as normal vector field.
Recall that $\sff$ is a tensor in $e,e'$ and $\nu$, as it is apparent from the equalities above. The mean curvature vector $H$ (or $H_M$) is the trace of the second fundamental form $\sff$ and is the only vector normal to $M$ such that
\[
H\cdot \nu = \sum_{i=1}^m \sff_\nu(e_i,e_i)\qquad   \text{for every } \nu\in T_x^\perp M\,. 
\]
Throughout the paper, without further mention, we will assume  that the flatness parameter $\delta$ is smaller than a positive constant $\delta_0$, chosen in terms of $m$ and $n$ only. The choice must guarantee that:
\begin{itemize}
    \item {the affine planes $x+T_x^\perp M$ do not intersect in $B^m_2\times B^n_{200}$ for distinct $x$'s;}
    \item any point $z\in B^m_{1.99}\times B^n_{200}$ has a unique closest point projection $p(z)\in M\cap \bC_{2}$;
    \item the map $p$ (often denoted $p_M$) is $C^1$ (and in fact $C^k$ if $M$ is of class $C^{k+1}$ for $k\in \mathbb N\setminus \{0\} \cup \{\infty, \omega\}$).
\end{itemize}
It is easy to see that an homothetic rescaling around a point of $M$ by a factor $\lambda\geq 1$ preserves the latter properties for the intersection of the rescaled $M$ with $\bC_2 (x)$.

It is useful to compute the differential of $p$ at a point $z$ (which does not necessarily belong to $M$).
\begin{lem}\label{lem:matrixp}
    Splitting $\R^{n+m}=T_{p(z)}M \times T_{p(z)}M^\perp,$ the Jacobi matrix of $Dp(z)$ can be written in block form as
    \[
    Dp(z)=\left[ 
    \begin{array}{c|c} 
    \big(\Id_m - \sff_{z-p(z)}\big)^{-1} & 0 \\ 
    \hline 
    0 & 0 
    \end{array} 
    \right].
    \]
    In particular,
    \[
    Dp(z)=\sum_{i,j=1}^m \big(\delta_{ij}+\sff_{z-p(z)}(e_i,e_j)\big) e_i\otimes e_j+ O(\delta^2 |z-p(z)|^2)\,.
    \]
\end{lem}
Second, we recall that $d_M^2$ is $C^k$ if $M$ is of class $C^{k+1}$ for $k\geq \mathbb N\setminus \{0\} \cup \{\infty, \omega\}$) and (when the manifold is $C^3$) its Hessian contains information on the curvature of $M$. 
\begin{lem}\label{hessdist} Assume that the frame $e_1,\dots,e_m$ at $p(z),$ diagonalizes the quadratic form $\sff_{{z-p(z)}/|{z-p(z)}|}$, and let $\kappa_1,\dots,\kappa_m$ be the respective eigenvalues. In particular, $\sum_{j=1}^m \kappa_j=H(p(z))\cdot \frac{z-p(z)}{|z-p(z)|}$. Then, with respect to this basis, \footnote{\, The fact that $\kappa_j$ are the eigenvalues of $\sff_{z-p(z)/|{z-p(z)}|}$ was conjectured, but not proved, in \cite[Remark 3.3]{ASlevel}. 
}
\begin{equation*}
    \tfrac12 D^2 d^2_M (z) = -{\sum_{j=1}^m \frac{d_M(z)\kappa_j(z)}{1-\kappa_j(z) d_M(z)} e_j\otimes e_j}+ {  \sum_{i={m+1}}^{m+n} e_i\otimes e_i}\,.
\end{equation*}
    
\end{lem}

\subsection{Sections of the normal bundle of $M$}
As we wish to approximate the varifold $V$ with a graph over $M$,  it is natural to recall the notion of section of the normal bundle of $M$. A ``normal section of $M$'' is simply a map $f\colon U\to \R^{m+n}$ such that $f(x)\in T_xM^\perp$ for all $x\in U$, and $U\subset M$ is some open set. We will then call the set
\[
\Gamma_f\defeq\{x+f(x) : x\in U\}
\]
the ``graph of $f$''.
Whenever $e\in T_xM$ we define $D_ef$ as
\[
D_ef(x)\defeq (D_e f_1(x),\ldots, D_e f_{m+n}(x)) \in \R^{n+m}\,,
\]
and 
\[
|D f(x)|^2\defeq \sum_{i=1}^m\sum_{j=1}^{n+m} |D_{e_i} f_j(x)|^2\,.
\]
While sections form a linear space, we warn the reader that $D_e f (x),$ will not be necessarily a normal section, just a map from $M$ to $\R^{m+n}$.

With these definitions, we can naturally define function spaces of sections of $M$ of various regularity and integrability, such as $C^{\alpha}$, $L^p$, $W^{1,p}$, and so on. For instance the $C^\alpha(U,\nu M)$ will denote the set of functions $f\colon U\to \R^{n+m}$ which are H\"older continuous with exponent $\alpha$ in $U\subset M$ and whose value at $x$ perpendicular to $T_xM$. More intrinsic definitions are possible, but for our purposes, since $M$ is a perturbation of an $m$-dimensional plane, they are all equivalent.


\subsection{The Jacobi operator}\label{jacop}
The ``stability operator of $M$'' or ``Jacobi operator of $M$'' is the linear differential operator on $M$ obtained by linearizing the minimal surface system over the normal bundle of a fixed minimal surface $M$, which maps normal sections of $M$ to normal sections of $M$. The operator can be written as 
\begin{equation*}
    \LL_M f = (\lap_M f)^\perp + 2 \sff_f:\sff\,,
\end{equation*}
where, if $f=(f_1,\ldots,f_{n+m})$, $\lap_M f$ is defined as\footnote{\, Using the intrinsic Laplacian of $\nu M$, usually denoted as $\lap ^\perp$, one has instead the expression $\LL_M f= \lap^\perp f+\sff_f\colon \sff,$ which often appears in literature, see e.g.\ \cite{Lawson}. More precisely, $\lap^\perp f $ is defined in such a way that
\[
\int_M f\cdot \lap^\perp \varphi = -\int_M (D f)^\perp \cdot (D \varphi)^\perp\qquad \text{for all } f\varphi\in C^1_c(M,\nu M)\,.
\]
So, in general, $\lap^\perp f$ is
\textit{different} from $(\lap_M f)^\perp$, which explains the factor 2 in \eqref{eq:defL_M}.}
\begin{equation}\label{eq:defL_M}
    (\lap_Mf )^\perp(x)= P_{T_x^\perp M}(\lap_M f_1(x),\ldots,\lap_M f_{n+m}(x)) \in T_xM^\perp\qquad \text{for }x\in M\, ,
\end{equation}
and $\sff_f:\sff\,$ is the normal section given by
\[
\sff_f:\sff (x)\defeq\sum_{i,j=1}^m \sff_f(e_i,e_j) \sff(e_i,e_j)\, .
\]
All these expressions do not depend on the choice of the orthonormal frame.

We can give a weak form of the stability operator through the bilinear operator
\begin{equation*}
    -\langle \LL_M f,g\rangle = \int_M  \big( D f :D g - 2\sff_f:\sff_g \big) = \int_M \sum_{i=1}^m \Big(\sum_{j=1}^{m+n} D_{e_i} f_j\cdot D_{e_i} g_j  - 2\sff_f:\sff_g\Big)\, ,
\end{equation*}
which is defined for $f,g \in H^1_0(B^M_1,\nu M)$.
From a general standpoint, $\LL_M$ is a strongly elliptic operator with smooth coefficients and enjoys regularity properties analogous to the usual $\lap_M$. We will need some results from the theory of linear elliptic systems.

\begin{lem}[$W^{2,p}$ bounds]\label{lem:W2pbounds}
Given a section $f\in L^\infty(B^M_1,\nu M)$, there exists a unique section $ u\in H^1_0(B^M_1,\nu M)$ such that
\begin{equation*}
    \int_M \big( D u:D\varphi  - 2\sff_u:\sff_\varphi\big) =\int_M f\cdot \varphi\qquad\text{for all }\varphi\in C^1_c(B^M_1,\nu M)\,.
    \end{equation*}
    Furthermore, $u\in C^1(B^M_1,\nu M)$ and 
    \[
    \|u\|_{C^1(B^M_1)} \le C \|f\|_{L^\infty(B^M_1)}\,,
    \]
    where $C$ depends only on $m,$ $n$ and $\delta_0$.
\end{lem}
 
\begin{lem}[Harmonic replacement]\label{lem:harmrepl}
Given a section $v\in \Lip(B^M_1,\nu M)$, there exists a unique section $ w\in H^1(B^M_1,\nu M)$ such that $w-v\in H^1_0(B^M_1,\nu M)$ and 
\begin{equation*}
    \int_M \big( D w:D \varphi -2 \sff_w:\sff_\varphi\big) =0\qquad\text{for all }\varphi\in C^1_c(B^M_1,\nu M)\,.
\end{equation*}
    Furthermore, $w\in C^\infty(B^M_1,\nu M)$ and 
\begin{equation}\label{eq:harmrep}
    \|w\|_{C^{2,1/2}(B^M_1)} \le C \|w\|_{L^\infty(B^M_1)}\le C^2\|v\|_{L^\infty(B^M_1)}\,,   
    \end{equation}
    where $C$ depends only on $m,$ $n$ and $\delta_0$.
\end{lem}
This shows in particular that if $\LL_M w=0$ and $w=0$ on the boundary, then $w=0$, provided $\delta_0$ is small.

\subsection{The minimal surface system and its linearization}\label{vfdscascdsa}
Given a normal section $f$ of $M$, let $\Gamma_f=\{(x+f(x):x\in M\}$ be its graph. If the size of $f$ (in a sense which will be specified below) is a small number $\eps>0$, we will then assert the well known-expansion expansion (again in a sense which will be specified further)
\begin{equation}\label{eq:Hexpansion}
    H_{\Gamma_f} = H_M + \LL_M f + O(\eps^2)\,,
\end{equation}
where $H_{\Gamma_f}(x)\in \R^{m+n}$ is the mean curvature vector of $\Gamma_f$ at the point $x+f(x)$. We will in fact need two versions of \eqref{eq:Hexpansion}, depending on how we measure the ``size'' of $f$. The first is suited for the case when $f$ is merely Lipschitz.
\begin{lem}\label{lem:geometriclin}
There are constants $C$ and $\ell>0$, depending only on $m$ and $n$ such that the following holds for every $\delta_0$ smaller than a positive geometric constant. Assume that $M$ is a $\delta_0$-flat minimal surface in $\bC_2$ and let $f\in \Lip (B^M_1, \nu M)$ with $\Lip (f) \leq \ell$.
Then the following identity holds for every $\varphi\in C^1 (B^M_1, \nu M)$
    \begin{equation}
       \Div_{\Gamma_{ f}}(\varphi\circ p) (x+f(x))=\big(D  f :D \varphi-2\sff_f:\sff_{ \varphi}\big)(x)+R(x)\qquad\text{for $\mathcal{H}^m$-a.e.\ $x\in B_1^M$}\,,\notag
    \end{equation}
    where the error term $R$ satisfies the estimate
    \begin{equation}
        |R|\le C| D\varphi|\big(\delta_0| f|+|D f|\big)^2\qquad\text{$\mathcal{H}^m$-a.e.\ on $B_1^M$}\,.\notag
    \end{equation}
    Moreover,
    \begin{equation}\label{notag}
        |D  f(x)|^2\le C(\delta_0^2 | f(x)|^2+|T_{x}M-T_{ x+{ f}(x)}\Gamma_{ f}|^2)\qquad\text{for $\mathcal{H}^m$-a.e.\ $x\in B_1^M$}\,.\notag
    \end{equation}
\end{lem}
In the second version of \eqref{eq:Hexpansion} we assume in addition that $f$ is $C^{2,\alpha}$. 
\begin{lem}\label{vefdscacdsac}
There are constants $C$ and $\ell>0$ such that the following holds for all $\delta_0>0$ small enough. Assume $M$ is a $\delta_0$-flat minimal surface in $\bC_2$ and let 
$f\in C^{2,1/2}(B_1^M,\nu M)$ be such that $\|f\|_{C^1}\leq \ell$ small enough. Then \begin{equation}
        \| H_{\Gamma_f} (x+f(x))- \LL_M f(x)\|_{C^{1/2}(B_1^M)}\le C \|f\|_{C^{2,1/2}(B_1^M)}^2\, .\notag
    \end{equation}
\end{lem}
From \eqref{eq:Hexpansion} one can solve the minimal surface equation with an inverse function argument since the linearized operator $\LL_M$ is invertible.
\begin{lem}[Classical solutions to the minimal surface system]\label{lem:solveMSS}
There are geometic constants $\eps_0>0$, $\delta_0>0$, and $C$, such that the following holds. 
Let $M$ be a $\delta$-flat minimal surface in $\bC_2$ with $\delta < \delta_0$ and assume that $ h\in C^{2,1/2}(B_1^M,\nu M)$ satisfies
\[
\LL_M h =0\quad\text{and}\quad \boldsymbol{\eps} \defeq \|h\|_{C^{2,1/2}(B^M_1)} \le \eps_0\,.
\]
Then there exists $ h'\in C^{2,1/2}(B_1^M,\nu M)$ such that its graph is a classical
minimal surface and 
\[
\|h-h'\|_{C^{2,1/2}(B^M_1)} \le C \boldsymbol{\eps}^2\,.
\]
In particular, $\Gamma_{h'}$ is $(\delta+C\boldsymbol{\eps})$-flat in $\bC_{1/2}$.
\end{lem}

\subsection{Some results on stationary varifolds}
We define the tilt-excess in the cylinder $\bC_r(x,\pi')$, with respect to $\pi$ as 
\begin{equation*}
    \bE(V,\bC_r(x,\pi'),\pi)\defeq\frac{1}{\omega_m r^m}\int_{\bC_r(x,\pi')} |T_y V-\pi|^2\,\dd\|V\|(y)\,.
\end{equation*}
We write $\bE(V,\bC_r(x,\pi))$ in place of $\bE(V,\bC_r(x,\pi),\pi)$ and, as usual, we omit $x$ when it is $0$ and $\pi$ when it is $\pi_0 = \R^m \times \{0\}$.

The following two results are taken from \cite{CCSvarifolds}, see also the references therein.  
 \begin{lem}\label{vfedacssc} Let $V$ be an $m$-dimensional integral varifold stationary in $\bC_1$ and let $Q\in\N$.
Let $0<r_1<r_2<1$ and let $\eta\in(0,\frac{1}{2})$. Assume that $\bE (V,\bC_1)$ is small enough, depending on $m$, $n$, $Q$, $r_1$, $r_2$ and $\eta$, and that $\frac{\|V\|(\bC_1)}{\omega_m}\le Q+\frac{3}{4}$. Then, there exists $Q'\in \{0,\dots,Q\}$ such 
	\begin{equation}\notag
		\frac{\|V\|(\bC_{r}(x))}{\omega_m r^m}\in (Q'-\eta,Q'+\eta)
        \qquad \mbox{for any $x\in \bC_{1}$ and any $r >r_1$ with $\bC_r(x)\subset \bC_{r_2}$\,.}
	\end{equation}
\end{lem}
\begin{thm}[Height bound]\label{height} Let $V$ be an $m$-dimensional integral varifold stationary in $\bC_1$, let $Q\in\N$, and let $ r\in (0,1)$.
Assume that  $\bE\defeq \bE(V,\bC_1)$  is small enough, depending on $m$, $n$, $Q$ and $r$, and that $ \frac{\|V\|(\bC_1)}{\omega_m}\le Q+\frac{1}{2}$. 
Then, there exist a constant $C$ depending on  $m$, $n$, $Q$ and $ r$,  and points  $y_1,\dots,y_Q\in \R^n$ such that 
\begin{equation}\notag
		\supp(V)\cap\bC_{ r}\subset \bigcup_{h=1,\dots,Q} \R^m\times B_{C\bE^{1/(2m)}}(y_h)\,.
\end{equation}
Moreover, if $\Theta(V,0)=Q$, then 
\begin{equation}\notag
		\supp(V)\cap\bC_{ r}\subset  \R^m\times B_{C\bE^{1/2}}(0)\,.
\end{equation}
\end{thm}

\section{Proof of the tilt excess inequality and of the height bound}

In this section we prove the tilt-excess inequality of Proposition \ref{prop:caccioppoli} and the height bound of Proposition \ref{prop:altezza}. We will use the following well-known fact in linear algebra, of which we include the elementary proof for the reader's convenience. Given a symmetric matrix $M\in \R^{d\times d}$ and a linear subspace $V\subset \R^d$ we will use  $\trace_{V}(M)$ for the trace of $p_V \cdot M$, where $p_V$ is the orthogonal projection onto $V$.

 \begin{lem}\label{tracce}
    Let $M\in \R^{d\times d} $ be symmetric and let $\lambda_1\le \ldots\le \lambda_d$ be its eigenvalues. Let $V$ be  a $k$-dimensional subspace of $\R^d,$ with $d\ge k$. Then $\trace_{V}(M)\ge \lambda_1+\ldots+\lambda_k$.
\end{lem}

\begin{proof}
We choose an orthonormal basis of $\R^d$, $e_1,\dots,e_d$, for which  $M$ takes the form $M=\sum_{i=1}^d \lambda_i e_i\otimes e_i$. Take also $\xi_1,\dots,\xi_k$ orthonormal basis for $V$. Therefore,
\begin{equation}\notag
     \trace_{V}(M)=\sum_{i=1}^d\sum_{j=1}^k \lambda_i e_i\otimes e_i:\xi_j\otimes\xi_j=\sum_{i=1}^d\sum_{j=1}^k \lambda_i (e_i\cdot\xi_j)^2\,.
\end{equation}
Now, write $\beta_i\defeq\sum_{j=1}^k (e_i\cdot\xi_j)^2 $. Notice that $\beta_i\in [0,1]$ with $\sum_{i=1}^{d}\beta_i=k$, as $e_1,\dots,e_d$ is an orthonormal basis. The claim then follows from the observation that
    \begin{equation*}
        \sum_{i=1}^{d} \lambda_i\beta_i\ge \sum_{i=1}^k\lambda_i\,,
    \end{equation*}
which is quickly proved:
\begin{equation*}
    \sum_{i=1}^k(1-\beta_i)\lambda_i\le \lambda_k \sum_{i=1}^k (1-\beta_i)=\lambda_k\sum_{i=k+1}^d\beta_{i}\le \sum_{i=k+1}^d \beta_i \lambda_i\,.\qedhere
\end{equation*}
\end{proof}

\subsection{Proof of Propositions \ref{prop:caccioppoli}} 
To simplify our notation we write $d$ for $d_M$ and $p$ for $p_M$. 
Fix $z\in \supp (V)$ and let $e_1,\ldots,e_m$ be an orthonormal basis for $T_{p(z)}M$, and let $e_{m+1},\dots,e_{m+n}$ be an orthonormal basis for $T_{p(z)}^\perp M$ with $e_{m+1}=Dd(z)=\frac{z-p(z)}{|z-p(z)|}$. Assume that $e_1,\dots, e_m$ diagonalize $\sff_{e_{m+1}}$, with eigenvalues $\kappa_1,\dots,\kappa_m$, and recall that $\kappa_1+\dots+\kappa_m=H(p(z))\cdot e_{m+1}$.
Recall that by Lemma \ref{hessdist}  we have
\begin{equation*}
    \frac12 D^2 d^2 (z) = \underbrace{\sum_{i=1}^m -\frac{d(z)\kappa_i(z)}{1-\kappa_i(z) d(z)} e_i\otimes e_i}_{\eqdef A} + \underbrace{  \sum_{i={m+1}}^{m+n} e_i\otimes e_i}_{\eqdef B}\,.
\end{equation*}
\\\medskip\textbf{Step 1.} We claim that, for each $z\in \supp(V)\cap C_{6/7},$
\begin{equation}\label{eq:cdsac}
\lap_V \tfrac{1}{2}d^2\ge\tfrac 14 |T_zV-T_{p(z)}M|^2- |H| d-2|\sff|d^2 . 
\end{equation}
Denote by $\xi_1,\dots,\xi_m$ be an orthonormal basis for $T_z V$.
    We compute 
    \begin{equation*}
        \Delta_V \frac{d^2}{2}= \sum_{j=1}^m D^2\frac{d^2}{2}(z):\xi_j\otimes \xi_j = \sum_{j=1}^m \Big(A+\frac{1}{2}B\Big): \xi_j\otimes \xi_j+ \frac{1}{2}\sum_{j=1}^m B: \xi_j\otimes \xi_j\,.
    \end{equation*}
    Now we deal with the first summand. The matrix $A+\frac{1}{2}B$ has eigenvalues $-\frac{d\kappa_1}{1-\kappa_1 d},\dots,-\frac{d\kappa_m}{1-\kappa_m d},\frac12,\dots,\frac12$; since $M$ is  sufficiently flat (depending upon a geometric quantity)  and $d\le 1$ on $\supp V$,
    \begin{equation*}
        \frac{d|\kappa_i(z)|}{|1-\kappa_i(z)d(z)|}\le\frac14\qquad \text{for all }  i=1,\ldots,m\,.
    \end{equation*}
    By Lemma \ref{tracce} we then have the lower bound
    \begin{align*}
    \sum_{j=1}^m \Big(A+\frac{1}{2}B\Big): \xi_j\otimes \xi_j
    &\ge\sum_{i=1}^m \frac{d\kappa_i}{1-\kappa_id}=  d\, H(p(z))\cdot D d-d^2 \sum_{i=1}^m\frac{\kappa_i^2}{1-\kappa_id}\\
    &\ge-  d|H(p(z))|-2|\sff(p(z))|^2d^2\,.
    \end{align*}
    For what concerns the second summand, clearly
    \begin{equation*}
    \frac 12\sum_{j=1}^m B: \xi_j\otimes \xi_j= \frac 12\sum_{i=m+1}^{m+n}\sum_{j=1}^m  (e_i\cdot\xi_j)^2=\frac14|T_zV-T_{p(z)}M|^2\,.
    \end{equation*}
We thus conclude \eqref{eq:cdsac}. Notice in passing  that 
\begin{equation}
    \frac12|T_zV-T_{p(z)}M|^2\ge \sum_{i=m+1}^{m+n}  \sum_{j=1}^m  (e_i\cdot\xi_j)^2\ge | D_V d|^2\,,
\end{equation}
as $D d\in T_{p(z)}^\perp M$.
\medskip\\\textbf{Step 2.} We prove  \eqref{eq:caccioppoli}.
    We test the first variation with $X= \zeta^2   D \frac{d^2}{2} $ and find
    \begin{align*}
        0 &= \int \Div_{V}\Big(\zeta^2  D \frac{d^2}{2}\Big)\, \dd\|V\|=2\int \zeta d D_V\zeta\cdot D_V d\, \dd\|V\|+\int \zeta^2 \Delta_V \frac{d^2}{2} \, \dd\|V\|\\
        &\ge -4 \int d^2 | D_V \zeta|^2\, \dd\|V\|-\frac {1}{4} \int \zeta^2| D_V d|^2\, \dd\|V\|+\int \zeta^2 \Delta_V \frac{d^2}{2} \, \dd\|V\|\\
        &\ge -4 \int d^2 | D_V \zeta|^2\, \dd\|V\|-\frac1 8 \int \zeta^2 |T_zV-T_{p(z)}M|^2\, \dd\|V\|+\frac14\int \zeta^2 |T_zV-T_{p(z)}M|^2\, \dd\|V\|\\
        &\qquad\qquad-\int \zeta^2 |H|d\, \dd\|V\|-2\int \zeta^2 |\sff|d^2\, \dd\|V\|\,.
        \end{align*}
    Rearranging terms and recalling $|\sff|\le \delta_0\le1$, we get \eqref{eq:caccioppoli}. We turn to the proof of \eqref{eq:L2Linf}.
    
\subsection{Proof of Proposition \ref{prop:altezza}} We start with a computation which is well-known when $V$ is a classical minimal surface (see e.g.\ \cite{ColdingMinicozzi}). Here we present a variation of the rigorous argument given in \cite{DPGS} when $V$ is an integral stationary varifold. 
With the same definition of $g_{r,s}$ as in \cite{DPGS}, we replace \cite[(2.6)]{DPGS} by
    \begin{equation}
        \begin{split}
                    \Div_V\Big(\frac{d^2}{2}  D g_{r,s}-g_{r,s} D\frac{d^2}{2}\Big)&=\frac{d^2}{2}\Delta_V  g_{r,s}-g_{r,s}\Delta_V \frac{d^2}{2}\\
        &\le -\frac{1}{2}\frac{d^2\mathbf{1}_{\B_r(0)}}{\omega_m r^m} +\frac{1}{2}\frac{d^2\mathbf{1}_{\B_s(0)}}{\omega_m s^m}-g_{r,s}(-|H| d-2|\sff|d^2)\,, \label{e:to-integrate}
                \end{split}
    \end{equation}
    where $\mathbf{1}_E$ is the indicator function of a $E$ and we have used the inequality $\Delta_V g_{r,s}\le -\frac{\mathbf{1}_{\B_r(0)}}{\omega_m r^m}+\frac{\mathbf{1}_{\B_s(0)}}{\omega_m s^m}$. Therefore, if we set 
    \[
    I(r)\defeq \frac{1}{\omega_m r^m}\int_{\B_r}\frac{d^2}{2}\, \dd\|V\|\,,
    \]
    we integrate the inequality \eqref{e:to-integrate} in $\B_s$ to obtain
    \begin{align*}
        I(s)-I(r)&\ge -\int g_{r,s}|H|d \, \dd\|V\|-2\int g_{r,s} d^2 \, \dd\|V\|\\
        &\ge -\int g_{r,s}|H|^2 \, \dd\|V\| -4\int g_{r,s}d^2 \, \dd\|V\|\,
    \end{align*}
    where we have used that $g_{s,r}$ vanishes outside $ \B_s$ and $|\sff|\leq C \delta_0 \leq 1$ (since $\delta_0$ can be chosen sufficiently small).
    Recalling from \cite{DPGS} the bound $\Big| \frac{ g_{r,s}}{s-r}\Big|\le Cr^{1-m}$, we can divide the above by $s-r$ and let $s\downarrow r$ to obtain, in the sense of distributions, 
    \begin{equation}\notag
        I'(r)\ge -\frac{C r}{r^m}\int_{\B_r(0)} |H|^2\, \dd\|V\|-C rI(r)\ge -C\|V\|(\B_1(0))r\sup_{\bC_1\cap M}|H|^2-C r I(r)\,.
    \end{equation}
    This means that $\big(I(r)+C  \|V\|(\B_1(0))\sup_{\B_1(0)\cap M}|H|^2\big) e^{Cr}$ is increasing. Now, $I(r)\rightarrow \Theta(V,0)\tfrac{1}{2}d^2(0)$ as $r\downarrow 0$. On the other hand $\Theta(V,0)\ge 1$ if $0\in \supp (V)$ (because $\Theta (V, \cdot)$ is upper semicontinuous and $\Theta (V, x)\geq 1$ for $\mathcal{H}^m$-a.e.\ $x\in \supp (V)$). Thus we obtain that, if $0\in\supp(V)$,
    \begin{equation}\notag
        d^2(0)\le C\Big( \frac{1}{r^m}\int_{\B_r(0)} d^2\, \dd\|V\|+\|V\|(\bC_1)\sup_{\bC_1\cap M}|H|^2\Big)\qquad \mbox{for every $r\in (0,1)$}\,.
    \end{equation}
    We repeat the argument with $r=\frac{1}{2}$ and $x\in \supp (V)\cap \bC_{1/2}$ to obtain
    \begin{equation}\notag
    d^2(x)\le C\Big( 2^m\int_{\B_{1/2}(x)} d^2\, \dd\|V\|+\|V\|(\bC_1)\sup_{\bC_1\cap M}|H|^2\Big)\le C\Big( \int_{\B_{1}(0)} d^2\, \dd\|V\|+\|V\|(\bC_1)\sup_{\bC_1\cap M}|H|^2\Big)\,.
    \end{equation}

\section{The Lipschitz approximation}

This section is dedicated to prove Theorem \ref{Lipapprox}.

\medskip
    
We first notice that, without loss of generality, we can assume that $\lambda$ is smaller than a positive geometric quantity, which will be chosen depending only on $m$, $n$ and $Q$.
We let $g:B_{10}\rightarrow\R^n$ be such that $M$ is the graph of $f$ and we recall that, since $M$ is $\delta$-flat with $\delta\leq \delta_0$, we can assume $\|g\|_{C^2}\le \delta$.

We next define the ``non-centered'' maximal function as
	\begin{equation}\notag
		B_3^M\ni x\mapsto \bmax\be(x)\defeq \sup_{x\in B_s^M(y)\subset B_{4}^M}\bE(V,\bC_s(y,T_yM))\,,
	\end{equation}
    where we notice that, if $\delta$ is small enough and $B_s^M(y)\subset B_4^M$, then $\bC_s(y,T_yM)\cap\supp(V)\subset\bC_{9/2}$. We remark also that, if $\delta$  is small enough, then, for any $B_{2s}^M(y)\subset B_{4}^M$, we have the inclusions 
    \begin{equation}\label{fdvsacasdca}
        \bC_s(y,T_yM)\cap\supp(V)\subset p_M^{-1}(B_{3s/2}^M(y))\cap\supp(V)\subset \bC_{2s}(y,T_yM)\cap\supp(V)\,.
    \end{equation}
    The first inclusion follows because we can assume that $p_M$ is $3/2$-Lipschitz, provided $\delta$ is small enough. We next argue for the second inclusion. Take $z\in p_M^{-1}(B_{3s/2}^M(y))\cap\supp(V)$ and set $z'\defeq p_M(z)\in B_{3s/2}^M(y)$. Then, 
    \begin{equation}\notag
        |p_{T_yM}(z-z')|\le |p_{T_yM}-p_{T_{z'}M}||z-z'|+|p_{T_{z'}M}(z-z')|\le s/2\,,
    \end{equation}
    again provided that $\delta$ is small enough.
    
    For the rest of the proof we set our choice of $K$, which is 
    \begin{equation}\notag
        K\defeq \{x\in B_1^M:\bmax\be(x)\le \lambda+\omega_m(Q+1)1000^m\delta^2\eqdef\tilde\lambda\}\,.
    \end{equation}
$K$ is clearly closed and up to removal of an $\mathcal{H}^m$-null set we can also assume that $\Theta$ is integer-valued over $p_M^{-1} (K)$ (note that in the statement $K$ is claimed to be just a Borel set). 
\medskip\\\textbf{Step 1.} Let $W$ be a stationary varifold in $\bC_{5}$, with $\frac{\|W\|(\bC_{5})}{\omega_m 5^m}\le Q+ 1$. Take $B^M_{\rho}(x)\subset B_4^M$  and assume the existence of a constant $\mu\in (0,1)$ such that  
	\begin{equation}\label{vfeadsac}
		\bE(W,\bC_s(y,T_yM))\le \mu \qquad\text{for every $B_\rho^M(x)\subset B_s^M(y)\subset B_{4}^M$}\,,
	\end{equation}
	We claim that, if $\mu$ and $\delta$ are small enough (depending upon $m$, $n$ and $Q$), there is $Q'\in\{0,\dots,Q\}$ such that
	\begin{equation}\label{brvfd}
		\frac{\|W\|(\bC_s(y,T_yM))}{\omega_m s^m}\in (Q'-\textstyle{\frac{1}{2Q}},Q'+\textstyle{\frac{1}{2Q}})\qquad\text{for every $B_\rho^M(x)\subset B^M_s(y)\subset B_{7/2}^M(0)$}\,.
	\end{equation}
	Notice that, in particular, if $x\in B_1^M$ is such that \eqref{vfeadsac} holds for every $x\in B_s^M(y)\subset B^M_{4}$, then \eqref{brvfd} holds for every $x\in B_s^M(y)\subset B^M_{7/2}$.

For the proof of the claim, we start from the following remark. For any $\eta_0\in (0,1)$ there is a choice of $\delta$ and $\mu$ such that, if $B_{s}^M(y)\subset B_{7/2}^M$, then
\begin{align}
     \bC_{s}(y,\pi_0)\cap\supp(V)&\subset\bC_{s+\eta_0}(y,T_y M)\cap\supp(V)\,,\notag\\
          \bC_{s}(y,T_y M)\cap\supp(V)&\subset \bC_{s+\eta_0}(y,\pi_0)\cap\supp(V)\,.\notag
\end{align}
We now notice, choosing $y=0$, that
\begin{equation}\notag
    \bE(W,\bC_4)\le \mu\, .
\end{equation}
 In particular, if $\mu$ is small enough,  Lemma \ref{vfedacssc} applies with parameters $r_1=1/128$, $r_2=15/16$, and $\eta=1/(4Q)$, yielding $Q'\in\{0,\dots,Q\}$. Thus, for every $B_r(z,\pi_0)\subset B_{15/4}$ with $r> \frac{1}{32}$, we have that $\frac{\|W\|(\bC_{r}(z,\pi_0))}{\omega_m r^m}\in (Q'-\frac{1}{4Q},Q'+\frac{1}{4Q})$. Notice in passing that, if \eqref{vfeadsc} holds for $W$ for some integer $Q''$, then $Q'=Q''$. 
By the considerations above, for every $B_s^M(y)\subset B_{{29}/{8}}^M$ with $s\ge \frac{1}{16}$ we conclude
\begin{equation}\label{freadscdcs}
\frac{\|W\|(\bC_{s}(y,T_y M))}{\omega_m s^m}\in (Q'-\frac{1}{2Q},Q'+\frac{1}{2Q})\, ,    
\end{equation}
 provided that $\eta_0$ and $\delta$ are small enough.
 Take any $B^M_\rho(x)\subset B^M_s(y)\subset B^M_{7/2}$ as above. If $s\ge \frac{1}{16}$, then \eqref{brvfd} follows from \eqref{freadscdcs}. In the case $s<\frac{1}{16}$, we argue as follows. For $l\ge 0$, consider the cylinders $\bC^l\defeq \bC_{1/2^{l+3}}(y,T_y M)$, notice that $B_{1/2^{l+3}}^M(y)\subset B_{{29}/{8}}^M$. We claim that, by induction, \eqref{brvfd} holds for $s=2^{-l-3}$, for every $l\ge 1$ with $B^M_\rho(x)\subset B^M_{2^{-l-3}}(y)$. The base case is observed above. For the inductive step at $l+1$,  we use (the scale-invariant form of) Lemma \ref{vfedacssc} on $ \bC^{l-1}(y)$, with the inductive assumption.
	For general $s$, there exists a unique $l'\in\N$, $l'\ge 1$, with $s\in [2^{-l'-4},2^{-l'-3})$, then we use again Lemma \ref{vfedacssc} on $ \bC^{l'-1}(y)$ together with the claim for $s=2^{-l'-3}$. 
\medskip\\\textbf{Step 2.} From now on we fix a $\mu$ such that the conclusions of the previous step apply and we impose that $\tilde\lambda\le \mu$. Now, since \textbf{Step 1} applies, recall that the integer $Q'$ as in \eqref{brvfd} for $Q=V$ must in fact be $Q$ by \eqref{vfeadsc}, as we have already noticed in the arguments at \textbf{Step 1}.
\medskip\\\textbf{Step 3.}
	Now take $x\in K$ and $x\in B_{s}^M(y)\subset B_3^M$  and consider $B^M_{7s/6}(y)\subset B^M_{7/2}$. We can assume that $\tilde\lambda$ is small enough (depending upon $m$, $n$ and $Q$) to apply   Theorem \ref{height} with $\bar r=\frac{6}{7}$.
	Hence, by  (the scale-invariant version of) Theorem \ref{height} applied to $\bC_{7s/6}(y,T_yM)$ (recalling \textbf{Step 1} for the bound on the measure), there exist $S_1^{y,s},\dots,S^{y,s}_{N^{y,s}}$ open and pairwise disjoint subsets of $T_y^\perp M$, with $\diam(S_h^{y,s})\le D\tilde\lambda^{1/(2m)}s$ (with $D\le C$, but independent of $y$ and $s$) and $N^{y,s}\le Q$, such that 
	\begin{equation}\label{vbgrfsdc}
	\supp(V)\cap\bC_{s}(y,T_yM) \subset\bigcup_{h=1,\dots,N^{y,s}} B_s(y,T_yM) + S_h^{y,s}\,.
	\end{equation}
	We also assume that for every $h$, $	\supp(V)\cap\bC_{s}(y,T_yM) \cap(B_s(y,T_yM) + S_h^{y,s})\ne \emptyset$.
	Set then $V_h^{y,s}\defeq V\mres (B_{s}(y,T_yM) + S_{h}^{y,s})$, notice that $V_h^{y,s}$ is a stationary varifold in $\bC_s(y,T_yM)$.
	Moreover, again by (the scale-invariant version of) \textbf{Step 1} (for both $V_h$ and $V$), we obtain that there exist $Q^{y,s}_h\in\{0,\dots,Q\}$ with   $\sum_h Q_h^{y,s}=Q$ and
	\begin{equation}\label{vfsedac}
		\frac{|V_h^{y,s}|(\bC_\rho(y',T_{y'}M))}{\omega_m \rho^m}\in (Q_h^{y,s}-\textstyle{\frac{1}{2Q}},Q_h^{y,s}+\textstyle{\frac{1}{2Q}})\qquad\text{for every $B^M_\rho(y')\subset B^M_{7s/10}(y) $ intersecting $K$}\,.
	\end{equation}
	Now fix $x\in K$ and let $s\in (0,1)$. Denote by $S_1^{x,s},\dots,S_{N^{x,s}}^{x,s}$ the sets as for \eqref{vbgrfsdc} for $\bC_{s}(x,T_xM)$, and choose points  $p_1^{x,s},\dots,p_{N^{x,s}}^{x,s}\in T_x ^\perp M$  with $p_h^{x,s}\in S^{x,s}_h$ and such that $\supp(V)\cap (B_{s}(x,T_xM)+\{p_h^{x,s}\})\ne\emptyset$. 
	Now observe that, thanks to \eqref{vfsedac}, 
	\begin{equation}\notag
		\dist_H\big(\{p_1^{x,s},\dots,p_{N^{x,s}}^{x,s}\},\{p_1^{x,\rho},\dots,p_{N^{x,\rho}}^{x,\rho}\}\big)\le 2D\tilde\lambda^{1/(2m)} s\qquad\text{for every }\rho\in (0,\textstyle{\frac{7s}{10}}]\,,
	\end{equation}
	so that we have a limit set $\{p^x_1,\dots,p^x_{N^x}\}$, where $N^x\le Q$, and $(x,p_h^x)\in\supp(V)$ for every $h$. Here and in the rest of this step, we write often $(x,y)$ for the point $(x+y)\in T_xM\times T_x^\perp M$. Later in the proof, especially in \textbf{Step 4}, we are going also to regard $p_h^x$ as a point of $\R^{m+n}$, hence, depending on the context, we use both notations $(x,p_h^x)$ and by $x+p_h^x$ to identify the same point.
	
	Now take $\bar s\in (0,1)$ small enough so that $B_{4D\tilde\lambda^{1/(2m)} \bar s}(p^x_h)$ are pairwise disjoint.  Take $s\in (0,\bar s)$ and consider $S_1^{x,s},\dots,S_{N^{x,s}}^{x,s}$. Take $h$ and consider $V^{x,s}_h$, notice that there exists a set $T_h^{x,s}$ with $\supp(V_h)\subset  B_{s}(x,T_xM)+ T_h^{x,s}$ and $\overline{T_h^{x,s}}\subset S_h^{x,s}$. By \eqref{vfsedac}, and exploiting $T_h^{x,s}$, we find a point  
    \[
    (x,\bar p)\in (\{x\}+\overline{T_h^{x,s}})\cap\supp(V^{x,s}_h)\subset\supp(V)\cap S_h^{x,s}\, ,
    \]
    which forces $\bar p=p^x_{h'}$ for some $h'$. We have thus seen that, for every $h$, there exists a unique $h'$ (as $|p^x_{h'}-p^x_{h''}|\ge 8D\tilde\lambda^{1/(2m)}\bar s>\diam(S^{x,s}_h)$ for $h''\ne h'$) such that $p^x_{h'}\in S^{x,s}_h$. Conversely, for every $k$, there exists a unique $k'$ such that $p^x_{k}\in S^{x,s}_{k'}$, as $p^x_{k}\in\supp(V)$. Hence we have a one-to-one correspondence between the points  $\{p^x_1,\dots,p^x_{N^x}\}$ and the sets $S_1^{x,s},\dots,S_{N^{x,s}}^{x,s}$. Hence, to ${p^x_h}$, we associate $Q_{h'}^{x,s}$ (for ${p^x_h}\in S^{x,s}_{h'}$) as the relevant integer as in \eqref{vfsedac}. Notice that \eqref{vfsedac} implies that this choice is well posed, i.e.\ independent of $s$, and recall that $\sum_h Q_h^{x,s}=Q.$ Therefore, we can define $f(x)\defeq \sum_h Q^{x,s}_h\a{p_h^{x}}\in\Iq(T_x^\perp M)$.  This implies clearly statement (a) of Proposition \ref{Lipapprox} and the estimate \eqref{cdsacads} on $K$. Notice also that the above argument implies $\supp(V)\cap(\{x\}\times T_x^\perp M )= \bigcup_h (x,p_h^x)$, in particular we have also covered statement (c).
            \medskip\\\textbf{Step 4.}
            In this step we show that 
            \[
            \cG(f(x_1),f(x_2))\le C(\tilde\lambda^{1/(2m)}+\delta)|x_1-x_2|\qquad \mbox{for every $x_1,x_2\in K$}\,.
            \]
            This proves, in particular, the resired Lipschitz estimate \eqref{Lipdcsa} for $f$ on $K$. 
            
            Set $\bar r\defeq |x_1-x_2|$ and let $\bar x$ be such that  $x_1,x_2\in B_{ 6\bar r/10}^M(\bar x)\subset B_{\bar r}^M(\bar x)\subset B_3^M$. 
			We use now \textbf{Step 3}, with the same notation, for $B_{\bar r}^M(\bar x)$. Notice first that 
            \begin{equation}\label{cdscsacs}
                (B_{\bar r/20}(x_1,T_{x_1}M)+ T_{x_1}^\perp M) \cap\supp(V) \subset B_{7/10\bar r}(\bar x,T_{\bar x}M)+ T_{\bar x}^\perp M\,.
            \end{equation}
            Indeed, take $z\in (B_{\bar r/20}(x_1,T_{x_1}M)+ T_{x_1}^\perp M) \cap\supp(V)$. We compute then 
            \begin{equation}\notag
                |p_{T_{\bar x}M}(z-\bar x)|\le |p_{T_{\bar x}M}-p_{T_{x_1}M}||z-\bar x|+|p_{T_{ x_1}M}(z- x_1)|+|p_{T_{ x_1}M}(\bar x- x_1)|\le \textstyle{\frac{7\bar r}{10}}
            \end{equation}
            provided that $\delta$ is small enough. 
            Consider the stripes $(S^{\bar x,\bar r}_{h})_h$ and the associated varifolds  $(V_h^{\bar x,\bar r})_h$. For any $h$, from \eqref{vfsedac} we know both that 
            \begin{gather}\label{vefaddvascac}
                \frac{|V_h^{\bar x,\bar s}|(\bC_{7/10\bar r}(\bar x,T_{\bar x}M))}{\omega_m (7/10\bar r)^m}\in (Q_h^{\bar x,\bar r}-\textstyle{\frac{1}{2Q}},Q_h^{\bar x,\bar r}+\textstyle{\frac{1}{2Q}})\,,
        \\\label{vefaddvascac1}
            \frac{|V_h^{\bar x,\bar s}|(\bC_{s}(x_1,T_{x_1}M))}{\omega_m s^m}\in (Q_h^{\bar x,\bar r}-\textstyle{\frac{1}{2Q}},Q_h^{\bar x,\bar r}+\textstyle{\frac{1}{2Q}})\qquad\text{for every }s\in (0,\textstyle{\frac{\bar r}{20}})\,.
            \end{gather}
            Now, take any $p^{x_1}_{k}$. 
            By \eqref{cdscsacs}, we know that $x_1+p^{x_1}_{k}\in B_{7/10 \bar r}(\bar x,T_{\bar x}M)+ S^{\bar x,\bar r}_{h(k)}$, in particular, $x_1+p^{x_1}_{k}\in\supp(V_{h(k)}^{\bar x,\bar s})$, thus defining the function $k\mapsto h(k)$.
            Now notice that, by the construction of \textbf{Step 3} and \eqref{vefaddvascac}, \eqref{vefaddvascac1}, if we fix $\bar h$, then
            \begin{equation}
            \sum_{h: h(k)=\bar h} Q_k^{x_1}=Q_{\bar h}^{\bar x,\bar r}\,.\notag
            \end{equation}
			Of course, the same consideration holds for $x_2$. Hence, 
            \begin{equation}\notag
                \cG(f(x_1),f(x_2))\le C\sup_{ h(k_1)=h(k_2)}|p^{x_1}_{k_1}-p^{x_2}_{k_2}|\,.
            \end{equation}           
             Take then $p^{x_1}_{k_1},p^{x_2}_{k_2}$ such that $h(k_1)=h(k_2)$, we compute
            \begin{align*}
                |p^{x_1}_{k_1}-p^{x_2}_{k_2}|&\le |p_{T_{\bar x}^\perp M }(p^{x_1}_{k_1}-p^{x_2}_{k_2})|+|p_{T_{\bar x}M }(p^{x_1}_{k_1}-p^{x_2}_{k_2})|\\
                &\le |p_{T_{\bar x}^\perp M }(x_1+p^{x_1}_{k_1}-x_2-p^{x_2}_{k_2})|+|p_{T_{\bar x}^\perp M }(x_1-x_2)|+|p_{T_{\bar x}M }(p^{x_1}_{k_1})|+|p_{T_{\bar x}M }(p^{x_2}_{k_2})|\\
                &\le  D\tilde\lambda^{1/(2m)}\bar r+C\delta\bar r^2+|p_{T_{\bar x}M }-p_{T_{x_1}M}||p^{x_1}_{k_1}|+|p_{T_{\bar x}M }-p_{T_{x_2}M}||p^{x_2}_{k_2}|\\
                &\le C(\tilde\lambda^{1/(2m)}+\delta)\bar r\,.
                \end{align*}
        Now that we have achieved the correct bounds,
        we can extend the Lipschitz function to $f:B_1^M\rightarrow \Iq(\R^{m+n})$ with $f(x)\perp T_xM$ for every $x$. This is done as in \cite{DS2}. Notice that this extension maintains \eqref{Lipdcsa} and \eqref{cdsacads}.
    \medskip\\\textbf{Step 5.} We prove \eqref{measureestimate}.
		Take $y\in B_1^M\setminus K$. Take $y\in B_{r'}^M(y')\subset B_{4}^M$ such that  $\bE(V,\bC_{r'}(y',T_{y'}M))>\tilde\lambda$ of almost maximal size, in the sense that 
		\begin{equation}\notag
			r'>1/2 \sup\{r'': y\in B_{r''}^M(y'')\subset B_{4}^M: \bE(V,\bC_{r''}(y'',T_{y''}M))>\tilde\lambda\}\,.
		\end{equation}
		If  $r'\ge \frac{1}{10}$, from
		\begin{equation}\notag
			\tilde\lambda<\bE(V,\bC_{r'}(y',T_{y'}M))\le \frac{\omega_m {5}^m \bE(V,\bC_{5},T_{y'}M)}{\omega_m (r')^m}\le 5^m\frac{\  \bE+\|V\|(\bC_{5})\delta^2}{\ (r')^m} \,
		\end{equation}
		we infer that $\lambda \bE^{-1}\le C$, hence \eqref{measureestimate} is trivial. Therefore, we can assume that $r'<\frac{1}{10}$, so that $B_{5r'}^M(y')\subset B_{{7}/{2}}^M$. This means that for every $B_t^M(z)$ with $B_{5r'}^M(y')\subset B_{t}^M(z)\subset B_4^M$ we have $\bE(V, \bC_{t}(z,T_z M))\le \tilde\lambda$. By \textbf{Step 1}, we have that 
		\begin{equation}\notag
			\frac{\|V\|(\bC_{5r'}(y',T_{y'}M))}{\omega_m (5r')^m}\le Q+\textstyle{\frac{1}{2}}\,.
		\end{equation}
		In particular, for every $y\in B_1^M\setminus K$, we can find $y\in B_{r'}^M(y')\subset B_4^M$ with $\bE(V,\bC_{r'}(y',T_{y'}M))>\tilde \lambda$ and $\|V\|(\bC_{5r'}(y',T_{y'}M))\le \omega_m(5r')^m(Q+1/2)$. Recalling \eqref{fdvsacasdca} for the last inequality,
  \begin{align*}
      \tilde\lambda& \le \frac{2}{\omega_m (r')^m}\int_{\bC_{r'}(y',T_{y'}M)} |T_z -T_{p_M(z)}M|^2\,\dd\|V\|(z)+\frac{2}{\omega_m (r')^m}\int_{\bC_{r'}(y',T_{y'}M)} |T_{p_M(z)}M -T_{y'}M|^2\,\dd\|V\|(z)\\
      &\le \frac{2}{\omega_m (r')^m}\int_{\bC_{r'}(y',T_{y'}M)} |T_z -T_{p_M(z)}M|^2\,\dd\|V\|(z)+2\omega_m(Q+1/2)5^m  \delta^2\,,
  \end{align*}
  so that, if we define the measure $\mu\defeq (p_M)_* \big(|T_z -T_{p_M(z)}M|^2\,\dd\|V\|(z)\big)$, 
  \begin{equation}\notag
      \lambda\le C\frac{1}{\omega_m (2r')^m}\int_{p_M^{-1}(B_{2r'}^M(y'))} |T_z -T_{p_M(z)}M|^2\,\dd\|V\|(z)=C\frac{\mu (B_{2r'}(y'))}{\omega_m (2r')^m}\,,
  \end{equation}
In particular, \eqref{measureestimate} follows from a standard covering argument.
\medskip\\\textbf{Step 6.} We prove  \eqref{densdvacs}.  We are going to use the same notation as in \textbf{Step 3}. Take $s$ so small that $B_{4D\tilde\lambda^{1/(2m)} s}(p_h^s)$ are pairwise disjoint, fix $z=(x,y)$ and $h$ such that $p_h^x=y$. 
If $s'\in (0,\frac{7s}{8})$, we have trivially
\begin{equation}\notag
	\frac{\|V_h^{x,s}\|(\B_{s'}(z))}{\omega_m {(s')}^m} \le \frac{\|V_h^{x,s}\|(\bC_{s'}(x,T_xM))}{\omega_m {(s')}^m}\,.
\end{equation} 
On the other hand, let $s'\in (0,\frac{7s}{8})$ be so small so that $\sqrt{1+D^2\tilde\lambda^{1/m}}s'<\frac{7s}{8}$. We have by \textbf{Step 3} that $\supp(V_h^{x,s})\cap \bC_{s'}(x,T_xM)\subset B_{s'}(x,T_xM)+ S_{h'}^{x,s'}$, with $\diam(S^{x,s'}_{h'})\le D\tilde\lambda^{1/(2m)}s'$, hence
\begin{equation}\notag
	\frac{\|V_h^{x,s}\|(\bC_{s'}(x,T_xM))}{\omega_m {(s')}^m}\le  \frac{\|V_h^{x,s}\|\Big(\B_{\sqrt{1+D^2\tilde\lambda^{1/m}}s'}(z)\Big)}{\omega_m {(s')}^m}=\frac{\|V_h^{x,s}\|\Big(\B_{\sqrt{1+D^2\tilde\lambda^{1/m}}s'}(z)\Big)}{\omega_m {(\sqrt{1+D^2\tilde\lambda^{1/m}}s')}^m} (1+D^2\tilde\lambda^{1/m})^{m/2}\,.
\end{equation}
We can let now $s'\downarrow 0$, recalling \eqref{vfsedac}, to infer that
\begin{equation}\notag
\Theta(z)\le Q^{x}_h+\textstyle{\frac{1}{2Q}}\quad\text{and}\quad	Q_h^x-\textstyle{\frac{1}{2Q}}\le \Theta(z)(1+D^2\tilde\lambda^{1/m})^{m/2}\,.
\end{equation}
Recalling that by assumption $\Theta(z)$ is an integer, we see that $\Theta(z)=Q_h^{x}$, provided $\tilde\lambda$ is small enough (depending only upon $m$, $n$ and $Q$).

\section{Proof of the decay lemma}

This section is devoted to prove Lemma \ref{decaylemma}. Keeping with the notation used so far, $C$ will denote a generic constant which depends only on $m$, $n$ and $Q$. 
Note that, up to translations we can assume $0\in M$ and keeping with the convention used so far we omit $0$ when it is the center of a ball, so that in particular $B_r^M$ denotes the geodesic ball in $M$ of center $0$ and radius $r$.

We collect first a couple of useful remarks.
By Propositions \ref{prop:caccioppoli} and \ref{prop:altezza}, we have
\[
\frac{1}{\omega_m 50^m}\int_{\bC_{50}} |T_z V-T_{p_M(z)}M|^2 \, \dd\|V\|(z) \le C\boldsymbol{\eps}^2\quad\text{and}\quad d_M\le C\boldsymbol{\eps} \text{ in }\bC_{50}\cap  \supp(V)\,.
\]
In particular,
\begin{equation}\notag
\supp (V)\cap \bC_{50}\subset \{(x,y):|x|<50,|y|<C(\delta+\boldsymbol{\eps})\}\,.   
\end{equation}
and
\begin{equation}\notag
\frac{1}{\omega_m 50^m}\int_{\bC_{50}} |T_z V-\pi_0|^2 \, \dd\|V\|(z)\le C(\delta^2+\boldsymbol{\eps}^2)\,.
\end{equation}
Therefore Remark \ref{portoavantiipotesi} follows from Lemma \ref{vfedacssc}, provided that $\eps_0$ and $\delta_0$ are small enough.

\subsection{The Lipschitz approximation and its estimates}\label{cfdesvac} 
We fix a $\lambda_0\in(0,1)$, whose choice will be given at the end of this step, depending only on  $m$, $n$ and $Q$.
We apply the Lipschitz approximation Theorem \ref{Lipapprox} (in scale-invariant form) with $\lambda = \lambda_0$. We thus have a $Q$-valued $C(\lambda_0+\delta^2)^{\sfrac{1}{2m}}$-Lipschitz normal section
\[
f\colon B^M_{10}\to \Iq(\R^{n+m})\,
\]
(in particular, for $x\in B^M_{10}$  we have $f(x)=\sum_{i=1}^Q \a{f_i(x)}$, with $f_i(x)\in T_x^\perp M$ for every $i=1,\dots,Q$) and an associated Borel set $K\subset B_{10}^M$ satisfying the estimate $|B_{10}^M\setminus K|+\|V\|(p_M^{-1}(B_1^M)\setminus K)\le C\lambda_0^{-1} \boldsymbol{\eps}^2$.
We require that $\delta_0$ is also small enough, depending on $m$, $n$ and $Q$, so that $\bar f$ will in fact be assumed to be $1$-Lipschitz (i.e.\ $C(\lambda_0+\delta_0^2)^{\sfrac{1}{2m}} <1$). 

The aim of this step is to show that, for any $\varphi\in C_c^1(B_{10}^M,\nu M)$,
\begin{equation}\label{toharm1}
		\bigg|\int_{B_{10}^M} \LL_M\bar f \cdot\varphi\,\dd\mathcal{H}^m \bigg|\le C \| D\varphi\|_{L^\infty}\boldsymbol{\eps}^2\,,
\end{equation}
where $\bar f$ denotes the average of $f$ ($\boldsymbol{\eta}\circ f$ in the notation of \cite{Almgren00} and \cite{DSq}), namely $\bar f(x)\defeq\tfrac{1}{Q}\sum_j f_j (x)$.
Along the way, we are going to show also that 
\begin{equation}\label{vefdsvadscvas}
    \int_{B_{10}^M}| D f|^2\,\dd\mathcal{H}^m\le C\boldsymbol{\eps}^2\,.
\end{equation}
We denote by $\Gamma_f$ the integral varifold given by the (multi)graph of $f$.
As we are going to use the claim of \cite[Lemma 1.1]{DSq} repeatedly, we recall its statement, in our context. Namely,  we have a sequence of pairwise disjoint Borel subsets of $B_{10}^M$, $\{E_h\}_h$, with $\mathcal{H}^m(B_{10}^M\setminus \bigcup_h E_h)=0$, such that 
\begin{enumerate}
	\item for every $h$, $f|_{E_h}=\sum_{i=1}^Q \a{f^h_i}$, with $f^h_i$ $C (\lambda_0+\delta^2)^{1/(2m)}$-Lipschitz on $E_h$ for every $i=1,\dots,Q$,
	\item for every $h$ and $i,i'\in \{1,\dots,Q\}$, either $f^h_i=f^h_{i'}$ on $E_h$, or $f^h_i\ne f^h_{i'}$ for every $x\in E_h$,
	\item for every $h$, $ D f(x)=\sum_{i=1}^Q \a{ D f^h_i}$ for a.e.\ $x\in E_h$.
\end{enumerate}
Now notice that for $\mathcal{H}^m$-a.e.\ $z\in\supp(\Gamma_f)$, if we write $f(p(z))=\sum_j \a{p_j}$, then
\begin{equation}\label{vfdsacxa}
	\Theta(\Gamma_f,z) =\sum_{j: p_j=z-p(z)} Q_j\,.
\end{equation}
This follows from \cite[Proposition 1.4]{DSsns} together with what remarked just after \cite[Definition 1.10]{DSsns}.
In particular, recalling also \eqref{densdvacs},
\begin{equation}\label{fadc}
	\Theta(\Gamma_f,z) =\Theta(z) \qquad\text{for $\mathcal{H}^m$-a.e.\ $z\in p^{-1}(K)\cap\supp(V)$}\,,
\end{equation}
and we recall also that, by differentiation of measures,
\begin{equation}\label{fadc1}
	T_z\Gamma_f=T_z V \qquad\text{for $\mathcal{H}^m$-a.e.\ $z\in p^{-1}(K)\cap\supp(V)$}\,.
\end{equation}

Let now $\varphi\in C_c^1(B_{10}^M,\nu M)$. With an harmless abuse, we test the variation of $V$ with the vector field $X(x)\defeq\varphi(p(z))$, so $\delta V(X)=0$, being $ V$ stationary. Therefore,
\begin{equation}\label{cdsaasc2}
\begin{split}
	\bigg|	\int \Div_{\Gamma_f}(X)\,\dd|\Gamma_f|\bigg|&\le C\| D \varphi\|_{L^\infty} \big(\|V\|(p^{-1}(B_{10}^M\setminus K))+Q\mathcal{H}^m\mres\Gamma_f( p^{-1}(B_{10}^M\setminus K) )\big)\\&\le C \| D \varphi\|_{L^\infty} \lambda_0^{-1} \boldsymbol{\eps}^2\,,
\end{split}
\end{equation}
where we used  \eqref{fadc} and  \eqref{fadc1} to deal with the portion on $p^{-1}(K)$, and the last inequality is by \eqref{measureestimate}.

Now, we concentrate on a set $E_h$, and fix $\hat f=f^h_i$ for some $i\in\{1,\dots,Q\}$. We will implicitly extend $\hat f$ to be defined on $B_{10}^M$, and still $C(\lambda_0+\delta^2)^{1/(2m)}$-Lipschitz. We define $\widehat \Gamma$ as the graph of $\hat f$.
By the area formula,
\begin{equation}\notag
	\int_{p^{-1}(E_h)}\Div_{\widehat\Gamma}(X)\,\dd\mathcal{H}^m\mres \widehat\Gamma=\int_{E_h}\Div_{\widehat\Gamma}(X)(x+\hat f(x)) J\hat f(x)\,\dd \mathcal{H}^m(x)\,,
\end{equation}
where $J_{\hat f}$ denotes the area factor of the map $x\mapsto x+\hat f(x)$. Now notice that
\begin{align}
\label{vefcdsxa}
	C\delta_0^2\boldsymbol{\eps}^2+C|T_{x}M-T_{x+\hat f(x)}\widehat\Gamma|^2\ge |D \hat f(x)|^2\qquad&\text{for $\mathcal{H}^m$-a.e.\ $x\in B_{10}^M$}\,,\\
        \label{areafactor}
	|J\hat f(x)-1|\le C|D \hat f(x)|^2\qquad&\text{for $\mathcal{H}^m$-a.e.\ $x\in B_{10}^M$}\,,
\end{align}
provided that   $\lambda_0$ and $\delta_0$ are small enough, depending on $m$ and $n$.
Indeed,  \eqref{vefcdsxa} is proved in Lemma \ref{lem:geometriclin}, whereas \eqref{areafactor}, follows from the Taylor expansion of the area factor. 
We thus deduce from the area formula and \eqref{areafactor} that
\begin{align*}
	&\bigg| \int_{E_h}\Div_{\widehat\Gamma}(X)(x+\hat f(x))\dd\mathcal{H}^m(x)- \int_{p^{-1}(E_h)}\Div_{\widehat\Gamma}(X)\,\dd\mathcal{H}^m\bigg| \\
    &\qquad\qquad\le    \bigg| \int_{E_h}|\Div_{\widehat\Gamma}(X)(x+\hat f(x))||J\hat f(x)-1|\dd\mathcal{H}^m(x)\bigg|
   \le C\| D\varphi\|_{L^\infty}\int_{E_h}| D \hat f|^2\,.
\end{align*}
We use now Lemma \ref{lem:geometriclin}  (provided that $\lambda_0$ and $\delta_0$ are small enough, depending on $m$ and $n$) to obtain that
\begin{equation}\label{vfjdnvsas}
\begin{split}
       & \bigg| \int_{E_h} D \hat f:D\varphi-2\sff_{\varphi}:\sff_{\hat f}\,\dd\mathcal{H}^m- \int_{E_h}\Div_{\widehat\Gamma}(X)(x+\hat f(x))\dd\mathcal{H}^m(x)\bigg|\\
   &\qquad\qquad\le C\| D\varphi\|_{L^\infty}\int_{E_h}\delta_0^2 |\hat f(x)|^2+|T_xM-T_{x+\hat f(x)}\widehat\Gamma|^2\,\dd\mathcal{H}^m(x)\,.
   \end{split}
\end{equation}
Now the choice of $\lambda_0$ can be fixed. Hence, all the dependencies on $\lambda_0$ can be absorbed in the constant $C$.

From the two set of inequalities above, recalling that $|\hat f|\le C\boldsymbol{\epsilon}$, we deduce that
\begin{align*}
   & \bigg| \int_{E_h} D \hat f :D\varphi-2\sff_{\varphi}:\sff_{\hat f}\,\dd\mathcal{H}^m- \int_{p^{-1}(E_h)}\Div_{\widehat\Gamma}(X)\,\dd\mathcal{H}^m\bigg|\\
   &\qquad\qquad\le C\| D\varphi\|_{L^\infty}\Big(\delta_0^2\boldsymbol{\eps}^2\mathcal{H}^m(E_h)+\int_{E_h}|T_xM-T_{x+\hat f(x)}\widehat\Gamma|^2\,\dd\mathcal{H}^m(x)+| D \hat f|^2\,\dd\mathcal{H}^m(x)\Big)\,.
\end{align*}
We recall that we chose $\hat f= f_i^h$. We next use \eqref{vefcdsxa}, the area formula (as $J {f_i^h}\ge 1/2$),  \eqref{vfdsacxa}, \eqref{fadc}, \eqref{fadc1}, and \eqref{measureestimate}, to compute
\begin{align}
       \int_{B_{10}^M}| D f|^2\,\dd\mathcal{H}^m&= \sum_{h}\sum_{i=1}^Q \int_{E_h}| D f_i^h|^2\,\dd\mathcal{H}^m\nonumber\\
    &\le \sum_h C\delta_0^2\boldsymbol{\eps}^2|\mathcal{H}^m(E_h)|+C\sum_h \sum_{i=1}^Q\int_{E_h}
	|T_{x}M-T_{x+f_{i}^h(x)}\Gamma_{f_i^h}|^2\dd\mathcal{H}^m(x)\nonumber\\
    &\le C\delta_0^2\boldsymbol{\eps}^2+C \sum_h\sum_{i=1}^Q\int_{p^{-1}(E_h)} |T_{p(z)}M-T_z\Gamma_{f_i^h}|^2\,\dd\mathcal{H}^m(z)\nonumber\\
    &\le C\delta_0^2\boldsymbol{\eps}^2+C \int_{p^{-1}(B_{10}^M)} |T_{p(z)}M-T_z\Gamma_f|^2\,\dd|\Gamma_f|(z)\nonumber\\
    &\le C\delta_0^2\boldsymbol{\eps}^2+  \int_{p^{-1}(B_{10}^M\cap K)}|T_{p(z)}M-T_{z}\Gamma_{ f}|^2\,\dd|\Gamma_f|(z)+\int_{p^{-1}(B_{10}^M\setminus K)}|T_{p(z)}M-T_{z}\Gamma_{ f}|^2\,\dd|\Gamma_f|\nonumber\\
&\le C\delta_0^2\boldsymbol{\eps}^2+   \int_{p^{-1}(B_{10}^M\cap K)}|T_{p(z)}M-T_{z}V|^2\,\dd\|V\|(z)+C|B_{10}^M\setminus K|\le C\boldsymbol{\eps}^2\,,\label{cdsaasc22}
\end{align}
which proves \eqref{vefdsvadscvas}.

We sum \eqref{vfjdnvsas} for $i=1,\dots,Q$ and then on $h$, using also (intermediate steps of) \eqref{cdsaasc22}, to obtain that
\begin{equation}
    \begin{split}\label{cdsaasc}
        &\bigg| Q\int_{B_{10}^M} \big(D \bar f :D\varphi-2\sff_{\varphi}:\sff_{\bar f}\big)\,\dd\mathcal{H}^m- \int_{p^{-1}(B_{10}^M)}\Div_{\Gamma_{ f}}(X)\,\dd|\Gamma_f|\bigg|\le  C\| D\varphi\|_{L^\infty}\boldsymbol{\eps}^2\,,
    \end{split}
\end{equation}
where we took into account  \eqref{vfdsacxa}.
Combining the outcomes of \eqref{cdsaasc2} and \eqref{cdsaasc}, we finally prove \eqref{toharm1}.

\subsection{Proof of Lemma  \ref{verfdsavdsc}}
To show that \eqref{verfdsavdsc1} follows, we compute
\begin{equation}\notag
    \int_{B_{10}^M}d_{\supp(V)}^2\,\dd\mathcal{H}^m\le \int_{K}d_{\supp(V)}^2\,\dd\mathcal{H}^m+\int_{B_{10}^M\setminus K}d_{\supp(V)}^2\,\dd\mathcal{H}^m\le C\int_{p^{-1}(K)} d^2_M \,\dd\|V\|+C\boldsymbol{\eps}^2\le C\boldsymbol{\eps}^2\,,
\end{equation}
where we used that  for every $x\in K$, $d_{\supp(V)}(x)\le |f(x)|\le C d_M(y)$ for some $y\in \supp(V)\cap p^{-1}(K)$, and the area formula.

Now, we use Proposition \ref{prop:altezza} with $M''$ to obtain that $\sup_{\bC_{1/2}(y)\cap\supp(V)}d_{M''}\le C\sigma$. Notice that for every $x\in K\cap \bC_{1/3}(y)$, there exists $x'\in\supp(V)$ with $|x-x'|\le C\boldsymbol{\eps}$ (by the $L^\infty$ bound for $f$), in particular $x'\in\supp(V)\cap \bC_{1/2}(y)$ (if $\eps_0$ is small enough), so that there exists $x''\in M''$ with $|x'-x''|\le C\sigma$. Therefore,
\begin{equation}\notag
    \int_{M\cap \bC_{1/3}(y)}d_{M''}^2\,\dd\mathcal{H}^m\le \int_{K\cap \bC_{1/3}(y)}d_{M''}^2\,\dd\mathcal{H}^m+\int_{B_{2}^M\setminus K}d_{M''}^2\,\dd\mathcal{H}^m\le C(\boldsymbol{\eps}+\sigma)^2\,.
\end{equation}
Now notice that
\begin{equation}\notag
    \|g-g''\|_{L^2(B_{1/4}(y))}^2\le C(\boldsymbol{\eps}+\sigma)^2\,.
\end{equation}
This follows from the inequality above together with Lemma \ref{lem:distlipgraphpc} below, and the area formula.
Therefore, \eqref{verfdsavdsc2} follows from the theory of elliptic PDEs. Indeed, $g-g''$ solves a linear elliptic system with smooth coefficients whose norms are bounded by functions of $\delta$.

\begin{lem}\label{lem:distlipgraphpc}
    Let $f\colon B_1\to\R^{n}$ be a $1$-Lipschitz function and let $\Gamma$ be its graph. Then
    \begin{equation*}
         d_\Gamma(z)\le |z-z_1-f(z_1)|\le 3 d_\Gamma(z)\qquad \text{for all }z=(z_1,z_2)\in B^m_{1/4}+ B^n_1\,.
    \end{equation*}
\end{lem}
We are not going to prove Lemma \ref{lem:distlipgraphpc}, as its proof is elementary.

\subsection{Elliptic regularization of the average} 
We proved in the previous step that the size of $\LL_M \bar{f}$ is bounded by $\boldsymbol{\eps}^2$ in a weak norm. It is convenient to replace $\bar{f}$ with a solution $\tilde{f}$ of $\LL_M \tilde{f}=0$, which we will show to be $L^2$-close to $\bar{f}$. Due to technical reasons, we actually get $O(\boldsymbol{\eps}^{3/2})$ closeness, which is still good enough for our purposes.

We in fact let $\tilde f \colon B^M_{10} \to \R^{n+m}$ as the normal section such that 
\begin{equation*}
    \LL_M \tilde f =0 \text{ in }B^M_{10} \quad \text{and}\quad \tilde f=\bar f \text{ on } \de B^M_{10}\,.
\end{equation*}
Applying the maximum principle to $|\tilde f |^2$ (see Lemma \ref{lem:harmrepl}) we find $|\tilde f|\le C\boldsymbol{\eps}$, which we claim implies
\begin{equation}\label{eq:hg}
    \|\tilde f - \bar f \|_{L^2(B^M_{10})} \le C\boldsymbol{\eps}^{3/2}\,.
\end{equation}
In order to prove \eqref{eq:hg}, we test the equation with the normal section $\varphi_0\in H^1_0(B^M_{10},\R^{m+n})$ such that $\LL_M \varphi_0= \bar f -\tilde f$ and use \eqref{toharm1} and $W^{2,p}$ estimates up to the boundary for linear elliptic systems (see Lemma \ref{lem:W2pbounds}). Hence,
\begin{align*}
    \|\tilde f-\bar f \|^2_{L^2(B^M_{10})}&= \langle \bar f -\tilde f,\LL_M  \varphi_0 \rangle = \langle \LL_M\bar f,  \varphi_0 \rangle \\
    &\le C\boldsymbol{\eps}^2 \| \varphi_0\|_{C^1(B^M_{10})}\le C\boldsymbol{\eps}^2\|\LL_M \varphi_0\|_{L^\infty(B^M_{10})}\le C\boldsymbol{\eps}^2\|\bar f- \tilde f\|_{L^\infty(B^M_{10})}\le C\boldsymbol{\eps}^3\,.  
\end{align*}
We denote by $\overline{\Gamma}$ the graph of $\bar f$ (that is the set
$\overline{\Gamma} \defeq\{ x+ \bar f(x) : x\in B^M_{10}\}$)
and similarly by $\widetilde{\Gamma}$ the graph of $\tilde f$. Then the $L^2$ bound \eqref{eq:hg} implies that
\begin{equation}\label{eq:gammabargammatilde}
\int_{\bC_{9}}d^2_{\widetilde{\Gamma}}\, \dd\|V\| \le C\int_{\bC_{9}}d^2_{\overline{\Gamma}}\,  \dd\|V\| + C\boldsymbol{\eps}^3\,.
\end{equation}
Indeed, as $\bar f$ is $1$-Lipschitz and for all $z\in \{|x|\le 9,|y|\le1\}$ one has
\[
d_{\widetilde{\Gamma}}(z)\le |z-p(z)-\tilde f(p(z))| \le  |z-p(z)-\bar f(p(z))| + |\tilde f(p(z))-\bar f(p(z))|\le 3d_{\overline{\Gamma}}(z)+|\tilde f(p(z))-\bar f(p(z))|\,,
\]
where $|z-p(z)-\bar f(p(z))| \le 3 d_{\overline{\Gamma}}(z)$ follows from Lemma \ref{lem:distlipgraph} below.
We integrate this inequality with respect to $\dd\|V\|$ on $E\defeq \Gamma\cap\bC_{9}$ and using the area formula on $p$, finding
\begin{align*}
    \int_{\Gamma\cap\bC_{9}}d^2_{\widetilde{\Gamma}}\, \dd\|V\|&\le 9\int_{\Gamma\cap\bC_{9}}d^2_{\overline{\Gamma}}\, \dd\|V\| + 9 \int_{\Gamma\cap\bC_{9}}|\tilde f(p(z))-\bar f(p(z))|^2 \, \dd\|V\|\\
    &\le 9\int_{\Gamma\cap\bC_{9}}d^2_{\overline{\Gamma}}\, \dd\|V\| +C \int_{p(\Gamma\cap\bC_{9})}|\tilde f(x)-\bar f(x)|^2 \,\dd\mathcal{H}^m\le 9\int_{\Gamma\cap \supp(V)}d^2_{\overline{\Gamma}}\, \dd\|V\| +C \boldsymbol{\eps}^3\,, 
\end{align*}
where we used $p(\Gamma\cap\bC_{9})\subset B^M_{10}$. On the piece $(\supp (V)\setminus \Gamma)\cap \bC_{9}$ we combine the $L^\infty$ estimate on $d_{\widetilde \Gamma}$ and the measure one:
\begin{align}\label{eq:pippo}
    \int_{\supp (V)\setminus\Gamma}d^2_{\widetilde{\Gamma}}\, \dd\|V\|&\le C \boldsymbol{\eps}^2 \cH^m(\supp (V)\setminus\Gamma ) \le C\boldsymbol{\eps}^4\,.
\end{align}
\begin{lem}\label{lem:distlipgraph}
    Let $f\colon B^M_{10}\to\R^{m+n}$ be a $1$-Lipschitz normal section and let $\Gamma$ be its graph. Then
    \begin{equation*}
         d_\Gamma(z)\le |z-p(z)-f(z)|\le 3 d_\Gamma(z)\qquad \text{for all }z\in B^m_{9}+ B^n_1\,.
    \end{equation*}
\end{lem}
Like we did for Lemma \ref{lem:distlipgraphpc}, we leave the elementary proof of Lemma \ref{lem:distlipgraph} to the reader.

\subsection{Updating the minimal surface} 
We construct now, in the smaller cylinder $\bC_{8}$ a true minimal surface $M'$ which is $O(\boldsymbol{\eps}^2)$-close to $\widetilde{\Gamma}$. 

By elliptic regularization, $\tilde f$ has size $O(\boldsymbol{\eps})$ in a strong norm in the interior of $B^M_{10},$ for example
\begin{equation}\notag
    \|\tilde f \|_{C^{2,1/2}(B^M_{5})} \le C\|\tilde f \|_{L^\infty(B^M_{10})}\le C\boldsymbol{\eps}\,.
\end{equation}
Since $\LL_M \tilde f=0,$ we solve the minimal surface system by the inverse function theorem of Lemma \ref{lem:solveMSS} to find a normal section $f'\colon B^M_{5}\to \R^{m+n}$ of $M$ such that its graph $M'$ is a $(\delta+C\boldsymbol{\eps})$-flat minimal surface and 
\[
\|f'-\tilde f\|_{C^{2,1/2}(B^M_{5})}\le C\boldsymbol{\eps}^2\,.
\]
In particular, arguing as in \eqref{eq:gammabargammatilde},
\begin{align*}
    \Big(\int_{ \bC_{1}} d^2_{M'}\, d\|V\|\Big)^{1/2} &\le C\Big(\int_{ \bC_{1}} d^2_{\widetilde{\Gamma}}\, d\|V\|\Big)^{1/2} +C\boldsymbol{\eps}^2\,.
\end{align*}
Combining this with \eqref{eq:gammabargammatilde}, we find
\begin{equation}\label{eq:pluto}
\Big(\int_{ \bC_{1}} d^2_{M'}\, \dd\|V\|\Big)^{1/2}\le\Big(\int_{ \bC_{1}} d^2_{\overline{\Gamma}}\, \dd\|V\|\Big)^{1/2} +C\boldsymbol{\eps}^{3/2}\,.  
\end{equation}
\subsection{Conclusion and the role of $\eta$}
In order to estimate the last integral, we use that $|d_{\overline{\Gamma}}|\le C\boldsymbol{\eps}$ on $V$ and the fact that $V$ coincides with the multigraph $\Gamma$ up to a set of measure $\boldsymbol{\eps}^2$ (as in \eqref{eq:pippo}). We find
\begin{align*}
    \Big(\int_{ \bC_{1}} d^2_{\overline{\Gamma}}\, d\|V\|\Big)^{1/2}& \le C\Big(\int_{ \Gamma\cap \bC_{1}} d^2_{\overline{\Gamma}}\, d\mathcal{H}^m\Big)^{1/2} +C\boldsymbol{\eps}^2\,.
\end{align*}
Now the parameter $\boldsymbol{\eta}$ comes into play, since $\overline{\Gamma}$ and $\Gamma$ differ precisely when $f$ is genuinely multivalued, which happens on a set of measure at most $\boldsymbol{\eta}$. Rigorously, let $F\subset B_{10}^M$ be defined as $F\defeq\big\{x:f(x)\ne Q \a{\bar f(x)}\big\}$. Then
\begin{align*}
    \mathcal{H}^m(F)\le  \mathcal{H}^m(F\cap K)+\mathcal{H}^m(F\setminus K)\le C\boldsymbol{\eta}+ C\boldsymbol{\eps}^2\,,
\end{align*}
where we used \eqref{densdvacs}  to deal with the first part and \eqref{measureestimate} to deal with the second.
It follows that  $\mathcal{H}^m(\{d_{\overline{\Gamma}}>0\}\cap \Gamma\cap \bC_1 )\le  C\boldsymbol{\eta}+C\boldsymbol{\eps}^2$. Hence we get 
\begin{equation*}
    \Big(\int_{ \Gamma\cap \bC_{1}} d^2_{\overline{\Gamma}}\, d\mathcal{H}^m\Big)^{1/2} \le C\sqrt{\boldsymbol{\eta}}\boldsymbol{\eps}+C\boldsymbol{\eps}^2\,,
\end{equation*}
which, together with \eqref{eq:pluto}, concludes the proof of \eqref{eq:contraction}.

\section{Proof of smooth rectifiability}

We are going to prove a stronger version of Theorem \ref{rectthm}, which reads as follows. 
\begin{thm}[Rectifiability]\label{rectthm1}
Let $V$ be an $m$-dimensional  stationary varifold in an open set $U$. Then, there exists a countable collection of Borel sets $(G_k)_k$ such that
\begin{equation*}
    \|V\|\Big(U\setminus \bigcup_k G_k\Big)=0\, ,
\end{equation*}
and, for every $k$, $G_k\subset M_k$ for a $m$-dimensional  $C^\infty$ submanifold $M_k$.  Moreover, for every $k$ and for every $z\in G_k$, we have
\begin{equation*}
     \sup_{\supp(V)\cap \B_{r}(z)} d_{M_k}=o(r^N)\qquad\text{as $r\downarrow 0,$ for every $N\in\N$}\,.
\end{equation*}
\end{thm}

Clearly the latter theorem imples both Theorem \ref{rectthm} and Corollary \ref{c:flat-high-order}.

\begin{lem}\label{rvefdsacdsca}
For every $m$, $n$, and $Q$ positive integers there are positive constants $\eta_1, \delta_1$, and $\eps_1$ with the following property.
Let $V$ be an $m$-dimensional stationary varifold in $\bC_{101}$ and let $Q\in\N$.
Let $G\subset\supp(V)\cap\bC_{1/20}$ be such that, for some modulus of continuity $\varpi:\R\rightarrow (0,\eta)$ with $\eta<\eta_1$,
\begin{equation}\label{eqqqmeno}
   \frac{\|V\|(\{\Theta_V<Q\}\cap \bC_{s}(z))}{\omega_m s^m}< \varpi(s)   \qquad\text{for all $s\in(0,100]$ and $z\in G$}\,.
\end{equation}
Assume moreover that, for every $z\in \bC_{1/20}\cap\supp(V)$,
    \begin{equation}\label{eqqq2}
    \supp(V)\cap \bC_{100}(z)\subset \{(x,y):|y|\le 1\},\quad  Q-\textstyle{\frac{1}{2}} <\frac{\|V\|(\bC_{30})(z)}{\omega_m(30)^m} <Q+\textstyle{\frac{1}{2}},\quad \frac{\|V\|(\bC_{100}(z))}{\omega_m100^m} \le Q+\textstyle{\frac{1}{2}}\,.
        \end{equation}
   Assume finally that there exists a minimal surface $M$ such that, for every $z\in \supp(V)\cap\bC_{1/20}$, $M$ is $\delta$-flat with $\delta<\delta_1$ in $\bC_{200}(z)$ and 
    \begin{equation}\label{eqqq3}
    \boldsymbol{\eps} \defeq \Big(\frac{1}{\omega_m100^{m+2}}\int_{\bC_{100}} d^2_M  \,\dd\|V\|\Big)^{1/2} \le \eps_1\,.
    \end{equation}
Then, for every $l\in\N$ and for every $z=(x,y)\in G$, there exists a polynomial $p_{z,l}:\R^m\rightarrow\R^n$ of degree $l$ with $p_{z,l}(x)=y$ satisfying the following. For every $z_1=(x_1,y_1),z_2=(x_2,y_2)\in G$, for every $l\in \N$ and multi index $\alpha$ with $|\alpha|\le l$,
    \begin{equation}\label{vefdscas}
|D^\alpha p_{z_1,l}(x_2)-D^\alpha p_{z_2,l}(x_2)|    \le C_{l,\varpi}|x_1-x_2|^{l-|\alpha|+1}\,, \end{equation}
where $C_{l,\varpi}$ depends on $m$, $n$, $Q$ and $\varpi$.
Moreover, for any $l'\le l$,  $p_{z,l}$ and $p_{z,l'}$ agree up to the terms of degree $l'$.
Finally, for every $z_0=(x_0,y_0)\in G$ and $l\in\N$,
\begin{equation}\label{fdaddcs}
    \sup_{z=(x,y)\in\supp(V)\cap \bC_{r}(x_0)}|y-p_{z_0,l}(x)|\le  C_{l,\varpi} r^{l+1}\quad\text{for every }r\in (0,1)\,.
\end{equation}
\end{lem}

\begin{proof}[Proof of Theorem \ref{rectthm1}]
Let $\eta_1,\delta_1$, and $\eps_1$ be the parameters given by Lemma \ref{rvefdsacdsca}.
We recall that for $\|V\|$-a.e.\ $z$, $\Theta_V(z)\in\mathbb{N}$ and the tangent to $V$ is 
$\Theta_V(z)\a{\pi_z}$, for some $m$-plane $\pi_z$. It is then enough to  fix $Q\in\mathbb{N}$ and prove the claim on the set 
   \begin{equation}\notag
       G\defeq \{z:\Theta_V(z)=Q,\ \exists\text{$s_0\in(0,1)$ and $\varpi:\R\rightarrow (0,\eta)$ such that \eqref{eqqqmeno} holds at $z$ for every $s\in (0,r_0)$}\}\,.
   \end{equation}
   Now, for a.e.\ $z\in G$, the tangent to $V$ is $Q\a{\pi_z}$. Fix any such $z$.
   In particular, the rescaled varifolds $V_r$ converge to $Q\a{\pi_z}$ (and the support of $V_r$ converge to $\pi_z$ in the Kuratowski sense). For definiteness, we assume that $\pi_z=\pi_0$.
   We thus see that, for $r$ small enough, $\hat V_r\defeq V_r\res \{(x,y):|y|<1\}$ is a stationary varifold satisfying the assumptions of Lemma  \ref{rvefdsacdsca} with $\hat G\defeq G\cap \bC_{1/20}$ in place of $G$ and  $\pi_0$ in place of $M$. A quick way to see \eqref{eqqq2} is by using Lemma \ref{vfedacssc}.  Now we apply Lemma  \ref{rvefdsacdsca} and obtain the polynomials $p_{z,l}$. Notice that \eqref{vefdscas} implies in particular that the projection $\mathbf{p}_0:\hat G\rightarrow\pi_0$ is injective. We then take, for $x\in \mathbf{p}_0(\hat G)$, the unique point $z_x$ with $z_x\in\hat G$ and $\mathbf{p}_0(z_x)=x$.
   Hence, for every $x\in \mathbf{p}_0(\hat G)$ and $l\in\N$, there exists a unique polynomial $p_{x,l}\defeq p_{z_x,l}$. The main claim of Lemma \ref{rvefdsacdsca}, i.e.\ \eqref{vefdscas}, implies that these polynomials satisfy the compatibility conditions required to apply the main theorem of Whitney in \cite{WhitneyExt}, thus obtaining a map $p\in C^\infty(B_1,\R^n)$ such that, for every $x\in \mathbf{p}_0(\hat G)$, and $\alpha$ with $|\alpha|\le l$, $D^\alpha p (x)=D^\alpha p_{x,l}(x)$, i.e.\ $p_{x,l}$ is the Taylor polynomial of degree $l$ for $p$ at $x$.
   We remark that, for $\alpha=(0,\dots,0)$, this reads as $(x,p(x))=z_x$, in particular,  $\hat G\subset \Gamma_p$  is a subset of a smooth submanifold. This proves the first part of the theorem. 

   Fix now $z_0=(x_0,y_0)\in \hat G$ and $N\in\N$. Let $l\in\N$. Recall \eqref{fdaddcs}, i.e.\
\begin{equation}\notag
    \sup_{z=(x,y)\in\supp(\hat V)\cap \bC_{r}(x_0)}|y-p_{x_0,l}(x)|\le  C_{l,\varpi} r^{l+1}\qquad\text{for every }r\in (0,1)\,.
\end{equation}
Hence, for every $r\in (0,1)$,
\begin{equation}\notag
     \sup_{z\in\supp(\hat V)\cap \bC_r(x_0)}d_{\Gamma_{p}}(z)\le\sup_{z=(x,y)\in\supp(\hat V)\cap \bC_{r}(x_0)}|y-p(x)|\le  \sup_{x\in B_r(x_0)}|p(x)-p_{x_0,l}(x)|+C_{l,\varpi} r^{l+1}\,,
\end{equation}
so that  the conclusion follows from Taylor's Theorem.   
\end{proof}
\begin{proof}[Proof of Lemma  \ref{rvefdsacdsca}]
The proof of this lemma is a careful iteration of Lemma \ref{decaylemma}. For this purpose we fix $\varepsilon_0$, $\delta_0$, and $\eta_0$ so that the latter applies.

\medskip

\textbf{Step 1}. Set $r_k\defeq 100^{1-k}$, for $k\ge 0$. Set $\eps(r_0)=\eps$ and, for $k\ge 0$, $\eps(r_{k+1})\defeq \bar C(\sqrt{\eta(r_k)}+\sqrt{\eps(r_k)})\eps(r_k)$. Moreover, set $\delta(r_0)=\delta$ and, for $k\ge 0$, $\delta(r_{k+1})\defeq\delta (r_k)+\bar C \eps(r_k)$. Here, the geometric constant $\bar C$ is the geometric constant that appears in the estimates of Lemma \ref{decaylemma}.
Now, if $\eps,\eta,\delta$ are small enough (depending also on $\bar C$), we can ensure what follows:
\begin{itemize}
    \item $\eps(r_k)\le \eps_0$ for every $k$,
    \item there exist constants $D_{l,\varpi}$ independent of $k$, but depending upon $m$, $n$, $Q$ and $\varpi$, such that $$\eps(r_k)\le D_{l,\varpi}r_k^{l}\qquad\text{for every }k\,,$$
    \item $\delta(r_k)\le\delta_0 $ for every $k$. 
\end{itemize}
We are going to exploit these relations throughout.
\medskip

\textbf{Step 2}. Take now any $z=(x,y)\in G$ and consider the cylinders $\bC_{r_k}(x)$. Set $M^0_z\defeq M$ (which is independent of $z$),   $M^0_z$ can be written as $\delta(r_0)$-flat graph of $g_z^0:B_{2r_0}(x)\rightarrow\R^n$, and has Taylor polynomial of degree $l$ at $x$, say $p_{z,l}^0$. Notice that among the assumptions of the lemma, there is that $V$ satisfies \eqref{eqqq2} and \eqref{eqqq3} on $\bC_{100}(x)=\bC_{r_0}(x)$. This is going to be the base step in an induction process.
%

For what concerns the inductive step, we argue as follows.
Assume that at step $k\ge 0$,  $V$ satisfies the scaled version of \eqref{eqqq2} on $\bC_{r_k}(x)$, and 
we have a  minimal surface $M_z^k$, which can be written as $\delta(r_k)$-flat graph of $g_z^k:B_{2r_k}(x)\rightarrow\R^n$, satisfying 
    \begin{equation}\label{vdsfvascs}
    \Big(\frac{1}{\omega_m r_k^{m+2}}\int_{\bC_{r_k}(x)} d^2_{M^k_z} \, \dd\|V\| \Big)^{1/2}\le \eps(r_k)\,.
    \end{equation}
    We denote by $p_{z,l}^k$ the  Taylor polynomial  of degree $l$ for $g_z^k$ at $x$.  
    
Then, we can recall \eqref{eqqqmeno} and apply Lemma \ref{decaylemma} to $\bC_{r_k}(x)$ and obtain  a  minimal surface $M_z^{k+1}$, which can be written as  a $\delta(r_k)+\bar C\eps(r_k)=\delta(r_{k+1})\le \delta_0$-flat graph of $g_z^{k+1}:B_{2r_{k+1}}(x)\rightarrow\R^n$, satisfying
    \begin{equation}\notag
    \Big(\frac{1}{\omega_m r_{k+1}^{m+2}}\int_{\bC_{r_{k+1}}(x)} d^2_{M^{k+1}_z} \, \dd\|V\| \Big)^{1/2}\le \bar C(\sqrt{\eta(r_k)}+\sqrt{\eps(r_k)})\eps(r_k)=\eps(r_{k+1})\le \eps_0\,.
    \end{equation}
    We denote by $p_{z,l}^{k+1}$ the  Taylor polynomial  of degree $l$ for $g_z^{k+1}$ at $x$.  
Moreover, we observe that the scaled version of \eqref{eqqq2} holds on $\bC_{r_{k+1}}(x)$ by Remark \ref{portoavantiipotesi}.
In passing, notice that Proposition \ref{prop:altezza} with \eqref{vdsfvascs} imply that 
\begin{equation}\label{vefdsvadscd}
    \sup_{\supp(V)\cap \bC_{r_k/2}(x)}d_{M_k}\le C\eps(r_{k})\,.
\end{equation}
In particular, as $z\in \supp(V)$, we have   
\begin{equation}\label{dwcasdac}
    g_{z}^{k}(x)\rightarrow y\qquad\text{as $k\rightarrow\infty$}\,.
\end{equation}

We can then iterate the procedure above and have the objects defined at every step $k$.
\medskip\\\textbf{Step 3}. Let $z_1=(x_1,y_1), z_2=(x_2,y_2)\in G$. Let $k\ge 1$ be such that $|x_1-x_2|< \frac{r_k}{10}$. We thus can apply the scale-invariant form of Lemma \ref{verfdsavdsc} to estimate
\begin{equation}\label{csadcacs}
\begin{split}
    \| D^\alpha g^{k-1}_{z_1}-D^{\alpha}g^{k}_{z_{2}}\|_{C^0(B_{r_{k}/10}(x_2))}&\le 
    \| g^{k-1}_{z_1}-g^{k}_{z_{2}}\|_{C^{|\alpha|}(B_{r_{k}/10}(x_2))}\le  C_{|\alpha|}   r_k^{1-|\alpha|} (\eps(r_{k-1})+\eps(r_k))\\&\le  C_l r^{1-|\alpha|}_k D_{l,\varpi} (r_{k-1}^l+r_{k}^l)\le 200^l C_lD_{l,\varpi} r_k^{1-|\alpha|+l}\le C_{l,\varpi}r_k^{1-|\alpha|+l}\,,
\end{split}
\end{equation}
for any multi-index $\alpha$ and $l\ge |\alpha|$. 

Now take $z_1=z_2=z=(x,y)\in G$, notice that the right-hand-side of \eqref{csadcacs} is a convergent sum, in particular, the following holds. First,
\begin{equation}\label{csadcacs1}
    \sup_{k\ge 0}\| g^{k}_{z}\|_{C^l(B_{r_{k}/10}(x))}\le E_l<\infty\,,
\end{equation}
for constants $E_l$ independent of $z$, where we used also that $g_z^0$ is independent of $z$. Also
\begin{equation*}
    |D^\alpha g^{k-1}_z(x) - D^\alpha g^{k}_z(x)|\le C_{l,\varpi} r_k^{1+l-|\alpha|}\qquad \text{for all }|\alpha|\le l\,.
\end{equation*}
This implies, for each multi index $|\beta|\le l$, the existence of the limit coefficients $p^\beta_{z}=\lim_k D^\beta g_{z}^k(x) /\beta!$, the fact that they do not depend on $l,$ the fact that $p^0_{z}=y$ (by \eqref{dwcasdac}) and the bounds $$|p^\beta_{z}|\le E_l,\qquad \Big|p^\beta_{z}-\frac{D^\beta g_{z}^k(x)} {\beta!}\Big|\le C_{l,\varpi} r_k^{1+l-|\beta|}\,,$$ 
where we have
\[
p^k_{z,l}(x') \defeq \sum_{|\beta|\le l}\frac{D_\beta g^k_{z} (x)}{\beta!}(x'-x)^\beta,\qquad p_{z,l}(x')\defeq\sum_{|\beta|\le l}p^\beta_{z}(x'-x)^\beta\,.
\]
Putting things together we get the following useful bound, for all $|\alpha|\le l$,
\begin{equation}\label{eq:terrible}
\begin{split}
    \|D^\alpha(p^k_{z,l} - p_{z,l})\|_{L^\infty(B_{r_k}(z))} &\le \sum_{|\beta|\le l-|\alpha|,\beta\ge\alpha } c_{\alpha,\beta} \Big|p^\beta_{z}-\frac{D^\beta g_{z}^k(x) }{\beta!}\Big|r_k^{\beta-\alpha} \\
    &\le C_{l,\varpi}\sum_{|\beta|\le l-|\alpha|,\beta\ge\alpha } r_k^{1+l-|\beta|} r_k^{|\beta|-|\alpha|} \le C_{l,\varpi}r_k^{1+l-|\alpha|}\,.
\end{split}
\end{equation}

Let $z_1=(x_1,y_1), z_2=(x_2,y_2)\in G$ and let $k\ge 1$ denote the largest integer such that $|x_1-x_2|<r_k/10$, this implies that  $r_k\le 1000|x_1-x_2|$. Let $l\ge 0$ and let $\alpha$ be a multi-index and with $|\alpha|\le l$.  We then have,
\begin{equation}\notag
\begin{split}
        |D^\alpha p_{z_1,l}^{k-1}(x_2)-D^\alpha p_{z_2,l}^{k}(x_2)|&\le |D^\alpha p_{z_1,l}^{k-1}(x_2)-D^\alpha g_{z_1}^{k-1}(x_2)|+|D^\alpha g_{z_1}^{k-1}(x_2)-D^\alpha g_{z_2}^{k}(x_2)|\\
        &\le |x_1-x_2|^{l-|\alpha|+1}\|D^\alpha g_{z_1}^{k-1}\|_{C^{l-|\alpha|+1}(B_{r_k/10}(x_1))}+\|D^\alpha g_{z_1}^{k-1}-D^\alpha g_{z_2}^k\|_{C^0(B_{r_k/10})(x_2)}\\
        &\le |x_1-x_2|^{l-|\alpha|+1}\|g_{z_1}^{k-1}\|_{C^{l+1}(B_{r_k/10}(x_1))}+\| g_{z_1}^{k-1}-g_{z_2}^k\|_{C^{|\alpha|}(B_{r_k/10})(x_2)}\\
        &\le |x_1-x_2|^{l-|\alpha|+1} E_{l+1}+200^{l}C_{|\alpha|}D_{l,\varpi} r_k^{1-|\alpha|+l}\\
        &\le \big(E_{l+1}+200^{l}C_{|\alpha|}D_{l,\varpi}1000^{1-|\alpha|+l})|x_1-x_2|^{l-|\alpha|+1}\\
        &\le C_{l,\varpi}|x_1-x_2|^{l-|\alpha|+1}\,,
\end{split}
\end{equation}
where we used \eqref{csadcacs} and \eqref{csadcacs1}. Now we  employ \eqref{eq:terrible} which gives for $i=1,2$,
\begin{equation*}
    |D^\alpha p_{z_i,l}^{k-1}(x_i) - D^\alpha p_{z_i,l}(x_i)| \le \|D^\alpha (p_{z_i,l}^{k-1} -  p_{z_i,l})\|_{L^\infty(B_{r_k}(x_1))}\le C_{l,\varpi}r_k^{1+l-|\alpha|} \le C_{l,\varpi}|x_1-x_2|^{l-|\alpha|+1}\,,
\end{equation*}
and this concludes the proof of \eqref{vefdscas}.
\medskip\\\textbf{Step 4}. We prove \eqref{fdaddcs}, fix $z_0=(x_0,y_0)\in G$. Let $k\in\N$. First, recalling Lemma \ref{lem:distlipgraphpc} ($g_{z_0}^k$ is $\delta_0$-Lipschitz), Taylor's Theorem and \eqref{vefdsvadscd} imply that \begin{equation}\notag
\begin{split}
          \sup_{z=(x,y)\in\supp(V)\cap \bC_{r_k/10}(x_0)}|y-p^k_{z_0,l}(x)|&\le C\eps(r_{k})+C\sup_{B_{r_k/10}(x_0)}|D^{l+1}g^k_{z_0}|r_k^{l+1}\\
          &\le CD_{l+1,\varpi}r_k^{l+1}+C E_{l+1} r_k^{l+1}\le C_{l,\varpi} r_k^{l+1}\,,
\end{split}
\end{equation}
where  we used also \eqref{csadcacs1}. If we  now use \eqref{eq:terrible} (with $\alpha=(0,\dots,0)$)  we obtain
\begin{equation}\notag
    \sup_{z=(x,y)\in\supp(V)\cap \bC_{r_k/10}(x_0)}|y-p_{z_0,l}(x)|\le C_{l,\varpi} r_k^{l+1}\,,
\end{equation}
which clearly implies \eqref{fdaddcs}.
\end{proof}
\appendix

\section{Proofs of the auxiliary results}
\subsection{Proofs of the results in Section \ref{vfsadcasddsvcca}}
\begin{proof}[Proof of Lemma \ref{lem:matrixp}]
    Choose a chart for $M$ around $p(z)$, say $\varphi$, with $\varphi(0)=p(z)$. Extend the orthonormal  vectors $e_1,\dots,e_m,e_{m+1},\dots,e_n$ to an adapted orthonormal frame, that is, an orthonormal frame whose first $m$ vectors are tangent to $M$. Define a chart $\Psi$ for the normal bundle $\nu M$ as follows: 
    \begin{equation}\notag
        \R^m\times \R^n\ni (x,(b_{m+1},\dots,b_{m+n}))\mapsto\Big(\varphi(x),\sum\nolimits_{i=m+1}^{m+n} b_i e_{i}(\varphi(x))\Big)\,.
    \end{equation} 
    Moreover let $F:\nu M\rightarrow\R^{m+n}$ be defined as $(y,\nu)\mapsto y+\nu$. We will first compute the differential of $F$ at $(0,(\bar b_{m+1},\dots \bar b_{m+n}))$, where $\sum_{i=m+1}^{m+n} \bar b_i e_{i}(z)=z-p(z)$. In our chart, we take a $C^1$ curve $\gamma_t=(x(t),(b_{m+1}(t),\dots b_{m+n}(t)))$ such that $\gamma_0=(0,\bar b_{m+1},\dots,\bar b_{m+n})$. Hence, using the short-hand notations $\bar b=(\bar b_{m+1},\dots,\bar b_{m+n})$ and $\dot b=(\dot b_{m+1},\dots \dot b_{m+n})$, we have that 
    \begin{align*}
        D(F\circ\Psi)(0,\bar b)[\dot x,\dot b]&=\partial_t(F\circ \gamma)_{|0}=D_{\dot x}\varphi+\sum \dot b_i e_{i}+\sum \bar b_i  D_{D_{\dot x}\varphi} e_{i}\\
       &=D_{\dot x}\varphi-\sff_{\sum \bar b_i e_i}(D_{\dot x}\varphi,\,\cdot\,)+\sum \bar b_i P_{T^\perp M}D_{D\varphi\dot x} e_i+\sum \dot b_i e_i\,,
    \end{align*}
    where all the sums are understood over $i=m+1,\dots,m+n$.
    We thus see that
    \[
    D(F\circ\Psi)(0,\bar b)[\dot x,\dot b]=\left[ 
    \begin{array}{c|c} 
    D_{\dot x}\varphi  - \sff_{z-p(z)}(D_{\dot x}\varphi,\,\cdot\,)& 0 \\ 
    \hline 
    * & \dot b^T 
    \end{array} 
    \right]\,.
    \]
    Now, notice that the map $z\mapsto (p(z),z-p(z))$ is the inverse of $F$, this means that, in our chart, $p(F\circ\Psi(x,b))=\varphi(x)$ for every $(x,b)$. Differentiating at $(0,\bar b)$, 
    $$Dp(F\circ\Psi (0,\bar b))\left[ 
    \begin{array}{c|c} 
    D_{\dot x}\varphi  - \sff_{z-p(z)}(D_{\dot x}\varphi,\,\cdot\,)& 0 \\ 
    \hline 
    * & \dot b^T
    \end{array} 
    \right]=
    \left[ 
    \begin{array}{c} 
    D_{\dot x}\varphi \\ 
    \hline 
     0
    \end{array} 
    \right]\,.$$ 
    This means that, for every $v\in \R^m$ and $w\in\R^n$, 
    $$Dp(x)\left[ 
    \begin{array}{c|c} 
    \Id_m  - \sff_{z-p(z)}& 0 \\ 
    \hline 
    * & \Id_n
    \end{array} 
    \right]
    \left[ 
    \begin{array}{c} 
    v \\ 
    \hline 
     w
    \end{array} 
    \right]=
    \left[ 
    \begin{array}{c} 
    v \\ 
    \hline 
     0
    \end{array} 
    \right]\,,$$ 
and thus the claim is proved.
\end{proof}

\begin{proof}[Proof of Lemma \ref{hessdist}]
    Since $d=|z-p(z)|$ we use the formula for $Dp(z)$ to compute $D\tfrac{1}{2}d^2=z-p(z).$
    We differentiate once again and obtain
    \begin{equation}\notag
        D^2\tfrac{1}{2}d^2=D(z-p(z))= \Id_{m+n}- Dp(z)=\left[ 
    \begin{array}{c|c} 
    \Id_m-\big(\Id_m - \sff_{z-p(z)}\big)^{-1} & 0 \\ 
    \hline 
    0 & \Id_n
    \end{array} 
    \right]\,,
    \end{equation}
    where we used Lemma \ref{lem:matrixp}.
    Now, 
    \begin{equation}\notag
        \Id_m - \sff_{z-p(z)}=\Id_m - |z-p(z)|\mathrm{diag}(\kappa_j)_{j=1,\dots,m}=\mathrm{diag}(1-d(z)\kappa_j)_{j=1,\dots,m}\,,
    \end{equation}
    so that we have proved the claim.
\end{proof}
\subsection{Proofs of the results in Section \ref{jacop}}
\begin{proof}[Proof of Lemma \ref{lem:W2pbounds}]
    For $\delta_0$ sufficiently small, the bilinear form $-\langle u,\LL_M u\rangle$ is elliptic in $H^1_0(B^M_1,\nu M).$ Existence and uniqueness of energy solutions to $\LL_M =f,$ for $f$ merely in $L^2(\nu M)$ follows by Riesz representation.

    Since the boundary $\de B^M_1$ is smooth, $W^{2,p}$ estimates hold up to the boundary:
    \begin{equation*}
            \|u\|_{C^1(B^M_1)} \le C \|u\|_{W^{2,m+1}(B^M_1)}\le C \|f\|_{L^\infty(B^M_1)}\,.\qedhere
    \end{equation*}
\end{proof}

\begin{proof}[Proof of Lemma \ref{lem:harmrepl}]
    The first inequality in  \eqref{eq:harmrep} is the interior Schauder regularity for strongly elliptic systems. For what concerns the second inequality, we claim that
    \begin{equation}\label{eq:subsol}
    \lap_M |w|^2 + C\delta_0^2 |w|^2 \ge0 \qquad\text{weakly in }B_1^M\,. 
    \end{equation}
    Indeed, we take as test function $\zeta w$ where $\zeta\in C^1_c(B_1^M)$ and $\zeta \ge 0$. On the one hand,
    \begin{align*}
        \int_M D w : D( \zeta w ) =2\int_M \sff_w:\sff_{\zeta w}\le C\delta^2 \int_M |w|^2 \zeta\,.
    \end{align*}
    On the other hand,
    \begin{align*}
        \int_M D w :D( \zeta w ) &\ge \sum_{i=1}^m \int_M (D_{e_i} w)\cdot (w D_{e_i}\zeta + \zeta D_{e_i}w)   \\
        &=\sum_{i=1}^m \int_M (D_{e_i}\zeta) w\cdot D_{e_i} w +\sum_{i=1}^m\int_M \zeta |D_{e_i} w |^2\ge \frac12 \int_M D\zeta \cdot D|w|^2\,,
    \end{align*}
    thus \eqref{eq:subsol} is satisfied weakly. Let $\varphi_1>0$ be the first eigenfunction of the Laplace-Beltrami operator
    \[
    -\lap_M \varphi_1 =\lambda_1\varphi_1\quad\text{and}\quad \varphi_1=0\text{ on }\de B_{7/6}^M\,.
    \]
    Now, as soon as $C\delta_0^2\le \lambda_1$, we claim that the function $|w|^2/\varphi_1$ cannot have interior maximum points in $B_1^M$. This proves that
    \begin{equation*}
        \sup_{B_1^M}|w|^2\le \frac{\sup_{\de B_1^M}\varphi_1}{\inf_{B_1^M} \varphi_1} \sup_{\de B_1^M}|w|^2\le C\sup_{\de B_1^M}|w|^2\,.
    \end{equation*}
    Let us check the claim. Set $\psi\defeq |w|^2$ and compute
     \begin{align*}
         \lap_M(\psi/\varphi_1)&= \frac{\varphi_1\lap_M \psi -\psi\lap_M \varphi_1}{\varphi_1^2}+2\frac{\psi| D \varphi_1|^2 - \varphi_1 D\varphi_1\cdot D \psi}{\varphi_1^3}\,.
     \end{align*}
     Now, if $x_0\in M\cap B_{1}^M$ is a putative interior maximum point for $\psi/\varphi_1$, we have
     \begin{equation*}
         \psi(x_0) D \varphi_1(x_0) = \varphi_1(x_0)  D \psi(x_0)\,,
     \end{equation*}
     so substituting
     \begin{align*}
         \lap_M(\psi/\varphi_1)(x_0)&\ge \frac{-C\delta_0^2 \psi\varphi_1 + \lambda_1 \psi\varphi_1}{\varphi_1^2}+\underbrace{2\frac{\psi| D \varphi_1|^2 - \varphi_1 D\varphi_1\cdot D \psi}{\varphi_1^3}}_{\text{ vanishes at $x_0$}}\ge (\lambda_1-C\delta_0^2)\frac{\psi(x_0)}{\varphi_1(x_0)}>0\,,
     \end{align*}
     a contradiction unless $\psi\equiv 0$. Thus the second part of \eqref{eq:harmrep} is proved.
\end{proof}
\subsection{Proofs of results in Section \ref{vfdscascdsa}}
\begin{proof}[Proof of Lemma \ref{lem:geometriclin}]
We suppress the subscript $f$ from $\Gamma_f$.
Let $e_1,\dots,e_m$ be an orthonormal frame for $M$. Define $F:M\rightarrow\R^{m+n}$ as $F(x)\defeq x+f(x)$, then set, for a.e.\ $z\in \Gamma$, $$w_i(z)\defeq (D_{e_i}F)(p(z))=(e_i+D_{e_i}f)(p(z))\,. $$ Notice that  $w_1,\dots,w_m$ form a basis of $T_{z}\Gamma$ for a.e.\ $z\in\Gamma$, provided that $\Lip (f)$ is small enough. We are going to use that 
\begin{equation}\label{vrfedsvasddcs}
    D_{w_i}w_j(F(x))=D_{e_i}(e_j+D_{e_j}f)(x)\,.
\end{equation}
Define also the matrix field $g_{i,j}\defeq w_i\cdot w_j$, and let $g^{i,j}$ be its inverse (which exists, if $\mathrm{LIP}(f)$ is small enough).

The following computations are carried for $\mathcal{H}^m$-a.e.\ $x\in M$.
First, for $i,j=1,\dots,m$,
\begin{equation}\label{vfedascxasc}
\begin{split}
        g_{i,j}(F(\,\cdot\,))&=\big(e_i\cdot e_j+D_{e_i}f\cdot D_{e_j}f+e_i\cdot D_{e_j}f+e_j\cdot D_{e_i}f\big)\\
        &=\big(\delta_{i,j}+D_{e_i}f\cdot D_{e_j}f-2\sff_{f}(e_i,e_j)\big)\,,
\end{split}
\end{equation}
where we used also the symmetry of the second fundamental form.
Now, 
\begin{equation}\notag
    \Div_{\Gamma} (\varphi\circ p)( F(\,\cdot\,))=\sum_{i,j=1}^m \big(g^{i,j}w_i\cdot D_{w_j} (\varphi\circ p)\big)(F(\,\cdot\,))=\sum_{i,j=1}^m g^{i,j}(F(\,\cdot\,))\big((e_i+D_{e_i }f)\cdot D_{e_j} \varphi \big)\,.
\end{equation}
We remark that by \eqref{vfedascxasc}, for $i,j=1,\dots,m$,
\begin{equation}\notag
    g^{i,j}=\delta_{i,j}-D_{e_i}f\cdot D_{e_j }f+2\sff_f(e_i,e_j)+O(|D f|^4+ \delta^2 |f|^2)\,,
\end{equation}
in particular,
\begin{equation}\notag
    g^{i,j}=\delta_{i,j}+2\sff_f(e_i,e_j)+O(|D f|^2+ \delta^2 |f|^2)\,.
\end{equation}
Hence, 
\begin{equation}\notag
         \Div_{\Gamma} (\varphi\circ p)(F(\,\cdot\,))=\sum_{i=1}^m (e_i+D_{e_i}f)\cdot D_{e_i}\varphi+\sum_{i,j=1}^m2\sff_f(e_i,e_j)(e_i+D_{e_i }f)\cdot D_{e_j}\varphi+ O\big((|Df|^2+\delta^2 f^2)|D\varphi|\big)\,.
\end{equation}
Now we compute 
\begin{equation}\notag
\begin{split}
     &\sum_{i=1}^m (e_i+D_{e_i}f)\cdot D_{e_i}\varphi+\sum_{i,j=1}^m2\sff_f(e_i,e_j)(e_i+D_{e_i }f)\cdot D_{e_j}\varphi\\
     &\qquad\qquad=-H\cdot\varphi+\sum_{i=1}^m D_{e_i}f\cdot D_{e_j}\varphi-2\sff_f:\sff_{\varphi}+2\sum_{i,j=1}^m\sff_f (e_i,e_j)D_{e_i}f\,\cdot D_{e_j}\varphi\\
      &\qquad\qquad=\sum_{i=1}^m D_{e_i}f\cdot D_{e_j}\varphi-2\sff_f:\sff_{\varphi}+O(\delta| D\varphi||f||D f|)\,.
         \end{split}
\end{equation}
This concludes the proof of the first part of the lemma. We have now to estimate $|Df|$.
Recalling that $P_{T\Gamma}=\sum_{i,j=1}^m g^{i,j}w_i\otimes w_j$ and $P_{TM}=\sum_{k=1}^m e_k\otimes e_k$,
\begin{equation}\notag
    \begin{split}
        \tfrac12|T_{F(x)}\Gamma-T_xM|^2&=\tfrac12|P_{T_{F(x)}\Gamma}|^2+\tfrac12|P_{T_{x}M}|^2-P_{T_{F(x)}\Gamma}:P_{T_{x}M}=m-\sum_{i,j,k=1}^m g^{i,j}w_i\cdot e_k w_j\cdot e_k\\
        &=m-\sum_{i,j,k=1}^m g^{i,j}(\delta_{i,k}+e_k\cdot D_{e_i}f)(\delta_{j,k}+e_k\cdot D_{e_j}f)\\
        &=m-\sum_{i,j=1}^m g^{i,j}(\delta_{i,j}+e_j \cdot D_{e_i} f+e_i\cdot D_{e_j} f+ P_{TM} D_{e_i}f \cdot P_{TM} D_{e_j}f )\\
        &=m-\sum_{i,j}^m \delta^{i,j}(\delta_{i,j}-2\sff_f (e_i,e_j)+ P_{TM} D_{e_i}f \cdot P_{TM} D_{e_j}f )\\
        &\qquad\qquad +\sum_{i,j}^m D_{e_i}f\cdot D_{e_j}f(\delta_{i,j}-2\sff_f (e_i,e_j)+ P_{TM} D_{e_i}f \cdot P_{TM} D_{e_j}f )\\
        &\qquad\qquad -\sum_{i,j}^m 2\sff_f(e_i,e_j)(\delta_{i,j}-2\sff_f (e_i,e_j)+ P_{TM} D_{e_i}f \cdot P_{TM} D_{e_j}f )\\
        &\qquad\qquad +O(|Df|^4+\delta^2|f|^2)\\
        &=-\sum_{i=1}^m| P_{TM}D_{e_i}f|^2+|D f|^2+O(\delta f|Df|^2+|Df|^4+\delta^2 f^2)\, ,
    \end{split}
\end{equation}
which proves that $|Df|^2\le \sum_{i=1}^m| P_{TM}D_{e_i}f|^2+ C(\delta_0^2 f^2+|T_{F(x)}\Gamma-T_xM|^2)$. Now,
\begin{equation}\notag
    \sum_{i=1}^m|P_{T M}D_{e_i}f|^2=\sum_{i,j=1}^m (D_{e_i} f\cdot e_j)^2=\sum_{i,j=1}^m (D_{e_i}  e_j\cdot f)^2=|\sff_f|^2=O(\delta^2 |f|^2)\,,
\end{equation}
which yields the conclusion.
\end{proof}
\begin{proof}[Proof of Lemma \ref{vefdscacdsac}]
    We use the same notation as in the proof of Lemma \ref{lem:geometriclin}. In particular, we use the  map $F(x)=x+f(x)$, the orthonormal frame $e_1,\dots,e_m$ for $M$ and the basis $w_1,\dots,w_m$ for $\Gamma$. We define as in Lemma \ref{lem:geometriclin} the matrices $g_{i,j}$  and $g^{i,j}$, and we recall \eqref{vfedascxasc}.  We are also going to use \eqref{vrfedsvasddcs}. In the following computations, we will use Einstein's convention and omit to write the summation over repeated indices. We also introduce  the error terms $E^{ij}\defeq g^{i,j}-\delta_{i,j}-2\sff_f(e_i,e_j)$.
    Recall that we have to compute the $C^{1/2}$ norm of 
    \begin{equation}\label{vefrdcvasa}
    \begin{split}
    H_{\Gamma}&=P_{T^\perp\Gamma} g^{i,j}D_{w_i} w_j= g^{i,j}D_{w_i} w_j-P_{T\Gamma } g^{i,j}D_{w_i} w_j\\
    &=g^{i,j} D_{w_i}w_j- g^{h,k} g^{i,j} w_h w_k\cdot D_{w_i}w_j\\
    &= (\delta_{i,j}+2\sff_f^{i,j}+ E^{i,j}) D_{e_i}(e_j+D_{e_j}f)\\&\qquad\qquad- (\delta_{h,k}+2\sff_f^{h,k}+ E^{h,k})(\delta_{i,j}+2\sff_f^{i,j}+ E^{i,j})(e_h+D_{e_h}f) (e_k+D_{e_k}f) \cdot D_{e_i}(e_j+D_{e_j}f)
    \end{split}
    \end{equation}
    where we used the explicit expression $P_{T \Gamma}=g^{h,k}w_h\otimes w_k$. To ease the notation, we define the vector field $v\defeq D_{e_i}e_i$, notice that $v\in TM,$ because $H_M=0$. 
    Notice that $E^{i,j}$ is made of terms of order at least  $2$ in $f$.
    We write the term of order $0$ in $f$ of \eqref{vefrdcvasa} as 
    \begin{equation}\notag
        D_{e_i}e_i-e_k e_k\cdot D_{e_i}e_i= v-P_{TM}v=P_{T^\perp M}v=H_M=0\,.
    \end{equation}
    We write the term of order $1$ in $f$ of \eqref{vefrdcvasa} as 
    \begin{equation}\notag
    \begin{split}
        &2\sff_f^{i,j}D_{e_i}e_j+ D_{e_i}D_{e_j}f\\
        &\qquad\qquad-2\sff_{f}^{h,k}e_he_k\cdot D_{e_i}e_i-2\sff_f^{i,j}e_k e_k\cdot D_{e_i}e_j-D_{e_k}f e_k\cdot D_{e_i}e_i-e_k D_{e_k}f\cdot D_{e_i}e_i-e_k e_k\cdot D_{e_i}D_{e_i}f\\
        &=2\sff_f^{i,j}D_{e_i}e_j-2\sff_f^{i,j}e_k e_k\cdot D_{e_i}e_j-2\sff_f^{h,k}e_he_k\cdot v+ D_{e_i}D_{e_j}f-D_{v}f-e_k D_{e_k}f\cdot v-P_{TM}D_{e_i}D_{e_i} f\\
        &=2\sff_f^{i,j}(P_{T^\perp M} D_{e_i}e_j)+2P_{TM}D_v  f+P_{T^\perp M}D_{e_i}D_{e_i}f-D_v f-e_k D_{e_k}f\cdot v\\
        &=2\sff_f:\sff_{\cdot}+2P_{TM}D_v f+P_{T^\perp M}\Delta f+P_{T^\perp M}D_v f-D_v f-e_k D_{e_k}f\cdot v\\
        &= \LL_M f+2P_{TM}D_v f+P_{T^\perp M}D_v f-D_v f-e_k D_{e_k}f\cdot v\\
        &=\LL_M f\,,
    \end{split}
    \end{equation}
    where we used the symmetry of the second fundamental form twice, precisely in $$\sff_f^{h,k}e_h e_k\cdot v=\sff_f^{k,h}e_h e_k\cdot v=D_{e_k}e_h\cdot f e_h e_k\cdot v=-D_{e_k}f\cdot e_h e_h e_k\cdot v=-P_{TM}D_v f$$ and  $$e_k D_{e_k}f\cdot v=-e_k D_{e_k} v \cdot f=-e_k\sff_f(e_k,v)=-e_k\sff_f(v,e_k)=-e_k D_v e_k\cdot f=e_ke_k\cdot D_v f=P_{TM}D_v f\,.$$
    %
    Hence, $H_\Gamma(F(x))= \LL_M f(x)+R(x)$, where the error therm $R$ satisfies
    $\|R\|_{C^{1/2}(B_1^M)}\le C\|f\|^2_{C^{2.1/2}(B_1^M)}$.
\end{proof}
\begin{proof}[Proof of Lemma \ref{lem:solveMSS}]
Fix an orthonormal frame $\{e_1,\ldots,e_{m}\}$ for $M$. 
If $v\colon M\to \R^{m+n}$ is a $C^2$ normal section of $M$, we denote by $\Gamma_v$ its graph, that is the manifold
\[
\Gamma_v\defeq\{ x+v(x) :x\in B^M_1\}\,.
\]
As in the previous proofs, the mean curvature vector of $\Gamma_v$ is
\begin{equation}\label{eq:Hv}
    H_{v}=g^{i,j} D_{w_i}w_j- g^{h,k} g^{i,j} w_h w_k\cdot D_{w_i}w_j\,.
\end{equation}
We remark that $H_v(x)\in \R^{m+n}$ lies in fact in the subspace $T_{x+f(x)}\Gamma_v^\perp.$
Here, $g^{ij}(x)$ are the entries of the inverse of the matrix $(w_i\cdot w_j)_{i,j}$ and $w_i(x+v(x))\defeq e_i(x)+D_{e_i} f(x)$.

Consider the map
\begin{align*}
    \Phi : X\to Y\times Z,\qquad 
    \Phi\colon v \mapsto ( H_v, v|_{\de B^M_1})\,,
\end{align*}
where $X,$ $Y$ and $Z$ are the following Banach spaces
\[
X\defeq C^{2,1/2}(B^M_1,\nu M ),\quad Y\defeq C^{1/2}( B^M_1,\R^{m+n}), \quad Z\defeq C^{2,1/2}(\de B^M_1,\nu M)\,.
\]
We claim that 
\begin{enumerate}
    \item $\Phi$ is of class $C^1(X,Y\times Z)$ around the point $v=0$;
    \item the derivative at $v=0$ is the Jacobi operator of $M$, that is $D_u\Phi(0)=(\LL_M u, u|_{\de B^M_1}),$ and it is invertible.
\end{enumerate}
Given the claims, the lemma follows by the inverse function theorem (since $\|h\|_X\le \eps_0$). Set indeed
\[
h'\defeq\Phi^{-1}(0,h|_{\de B^M_1})
\]
and combine Lemma \ref{vefdscacdsac} with the Lipschitz character of $\Phi^{-1}$:
\[
\|h-h'\|_X\le C\|H_h\|_{Y} \le C\|h\|_X^2\le C\eps^2\,.
\]

Now, item (1) follows inspecting the formula giving $H_v$, namely \eqref{eq:Hv}, which can be generally written as
\[
H_v(x)=F(x,v(x),Dv(x),D^2v(x))\,,
\]
with $F$ smooth and all arguments ranging in a compact domain. If $\Phi=(\Phi_Y,\Phi_Z)$, then
\[
D_\varphi \Phi_Y(v)=\frac{d}{dt}\Big|_{t=0} H_{v+t\varphi}(x) = \sum_{k=0}^2 F_{k+2}(x,v(x),Dv(x),D^2v(x))D^k\varphi(x)\,,
\]
where $F_j$ is a short-hand for derivatives of $F$ in the $j$-th entry.
Now using that $C^{1/2}$ is an algebra and that $F$ is smooth one concludes that if $v_k\to v$ in $X=C^{2,1/2}$ and $\|u_k\|_X\le 1,$ then
\[
\|D_{u_k}\Phi_Y(v_k)-D_{u_k}\Phi_Y(v)\|_{C^{1/2}} \to 0\,.
\]

Let us prove item (2). Lemma \ref{vefdscacdsac} shows that $D_u\Phi(0)=\LL_M u.$ 
The maximum principle for $\LL_M$ (see Lemma \ref{lem:harmrepl}) shows that $D\Phi(0)$ is injective. Schauder regularity and Lemma \ref{lem:W2pbounds} show that $D\Phi(0)$ is surjective.
\end{proof}

\end{document}